\documentclass[onecolumn]{article}

\usepackage{lineno,hyperref}

\usepackage{amsmath}
\usepackage{amsthm} 
\usepackage{amssymb}  
\usepackage{mathrsfs}
\usepackage{graphicx}
\usepackage{array}
\usepackage{xspace}		
\usepackage[english,ruled,vlined]{algorithm2e}	
\usepackage{color}

\usepackage{scalefnt}
\usepackage{todonotes}
\usepackage{subcaption}
\usepackage{geometry}
\usepackage{enumitem}
\renewcommand{\nu}{\vec{\mathbf{n}}}
\newcommand{\HH}{\mathcal{H}}
\newcommand{\scal}[2]{{\left\langle{{#1}\mid{#2}}\right\rangle}}
 
\newcommand{\menge}[2]{\big\{{#1}~\big |~{#2}\big\}} 
\newcommand{\Id}{\ensuremath{\operatorname{Id}}}
\newcommand{\minimize}[2]{\ensuremath{\underset{\substack{{#1}}}%
{\text{\rm minimize}}\;\;#2}}
\newcommand{\prox}{{\rm prox}}

\newcommand{\RR}{\ensuremath{\mathbb{R}}}
\newcommand{\RP}{\ensuremath{\left[0,+\infty\right[}}

\newcommand{\RPP}{\ensuremath{\left]0,+\infty\right[}}

\newcommand{\RX}{\ensuremath{\left]-\infty,+\infty\right]}}

\newcommand{\Fix}{\ensuremath{\text{\rm Fix}\,}}

\newcommand{\NN}{\mathbb{N}}

\newtheorem{proposition}{Proposition}

\newtheorem{problem}{Problem}
\newtheorem{remark}{Remark}
\newtheorem{example}{Example}

\newtheorem{lemma}{Lemma}

\graphicspath{
	{figures/},
	{figures/u-v/},
}
\makeatletter
	\def\input@path{%
		{figures/},%
	}
\makeatother

\title{Theoretical and numerical comparison of first order algorithms for 
cocoercive equations and smooth convex optimization}

\author{Luis M. Brice\~no-Arias\thanks{L. M. Brice\~no-Arias, Department of Mathematics, 
Universidad T\'ecnica Federico Santa Mar\'ia,
 Santiago, Chile.} \,and Nelly Pustelnik
\thanks{N. Pustelnik is with Univ Lyon, Ens de Lyon, Univ Lyon 1, CNRS, Laboratoire de Physique, Lyon, 69342, France.}}

\begin{document}
\maketitle

\begin{abstract}
This paper provides a theoretical and numerical comparison of 
classical {first-order} splitting methods for solving smooth convex 
optimization problems and cocoercive equations. {From}  a 
theoretical point 
of view, we compare convergence rates of gradient descent, 
forward-backward, Peaceman-Rachford, and Douglas-Rachford 
algorithms for minimizing the sum of two smooth convex functions 
when one of them is strongly convex. A similar comparison is given in 
the more general cocoercive setting under the presence of strong 
monotonicity and we observe that the convergence rates in 
optimization are strictly better than the corresponding rates for 
cocoercive equations for some algorithms.
We obtain improved rates with respect to the literature in several 
instances {by}  exploiting the structure of our problems. 
Moreover, we  
indicate { which algorithm has the lowest convergence rate 
depending on 
strong 
convexity and cocoercive parameters.}
{From}  a numerical point of view, we verify our theoretical results 
by implementing and comparing previous algorithms in well 
established signal and image inverse problems involving sparsity. 
We replace the widely used $\ell_1$ norm {with}  the Huber loss 
and we observe that 
fully proximal-based strategies have numerical and theoretical 
advantages with respect to methods using gradient steps. 
\footnote{This work is supported by Agencia Nacional de 
Investigaci\'on 
y Desarrollo (ANID) from Chile, under grant FONDECYT 1190871, by 
Centro de
Modelamiento Matem\'atico (CMM), ACE210010 and FB210005, 
BASAL funds for centers of excellence, and also by the ANR (Agence 
Nationale de la Recherche) from France
ANR-19-CE48-0009 Multisc'In. The authors acknowledge the support 
from 
LIA-MSD, CNRS-France.}
\end{abstract}

\textbf{Keywords: }
Proximal algorithms, convergence rates, cocoercive equations, 
smooth convex optimization, Huber loss, sparse inverse problems.

\section{Introduction}
\label{sec:intro}

The resolution of many signal processing problems relies on the 
minimization of a sum of {data-fidelity term and penalization}. 
This formulation can be encountered either in standard variational 
strategies~\cite{Pustelnik_N_20016_j-w-enc-eee_wav_bid}, mainly 
used in the past {60 years}, or into more recent deep-learning 
framework~\cite{Fessler_JA_2020_j-ieee-spm}.  

Formally, the associated
optimization problem writes
\begin{equation}
\label{e:proboptimintro}
\underset{x\in\HH}{\textrm{minimize}}\,f(x)+g(x),
\end{equation}
where $\HH$ denotes a real Hilbert space, and $f\colon\HH\to]-\infty,+\infty]$
and $g\colon\HH\to]-\infty,+\infty]$ are very often considered as 
proper lower semicontinuous convex functions.

{For} 
almost twenty years, a large panel of efficient first-order 
algorithms {has} 
been 
derived in order to solve \eqref{e:proboptimintro} under different
assumptions on functions $f$ and $g$ (see 
\cite{bauschke2011:Convex_analysis,Combettes_P_2007_inbook,
parikh2014proximal}  
for an exhaustive list). 
From stronger to weaker assumptions, the gradient method 
\cite{Cauchy,Haskell} is 
implementable if 
$f$ and $g$ are smooth, forward-backward splitting (FBS)
\cite{Combettes_P_2005_j-mms_sig_rpf,Mercier80} can 
be applied when 
either $f$ or $g$ is smooth, while Peaceman-Rachford 
splitting (PRS) \cite{LionsMercier} and  
Douglas-Rachford 
splitting (DRS) \cite{LionsMercier,Eckstein_J_1992_j-mp_dr, 
Combettes_P_2007_j-ieee-jstsp_dr} are applicable without any 
smoothness assumption. When a function is not smooth,
FBS, PRS, DRS use proximal (implicit) steps for the function, which 
amounts to solve a non-linear equation. Since solving a non-linear 
equation at each iteration can be computationally costly, a common 
practice is to choose gradient steps when the function is smooth.
However, nowadays there exists a wide class of functions 
whose proximal steps are explicit or easy to 
compute\footnote{See, e.g., http://proximity-operator.net/} and 
activating $f$ and $g$ via proximal steps can be advantageous 
numerically
\cite{combettesGlaudin}.
In this context, it becomes important to provide a theoretical 
comparison of algorithmic schemes involving gradient and/or proximal 
steps for solving \eqref{e:proboptimintro} and to identify which
algorithm is the most efficient depending on
the properties of $f$ and $g$. We focus our analysis on 
{first-order} methods when $f$ and $g$ are smooth and proximal 
steps 
of both functions are easy to compute. The theoretical 
analysis of several first-order methods in this context provides 
interesting insights of the structural properties of first-order 
algorithms to be considered in more general frameworks.

From the signal processing {user's} point of view, the choice of 
the most 
efficient algorithm for a specific data processing problem with the 
 form of \eqref{e:proboptimintro} is a complicated task.  
In order to tackle this problem,
the convergence rate is {a} 
useful tool in order 
to provide a theoretical comparison among algorithms. 
However, 
the theoretical
behavior of an algorithmic scheme may differ considerably from its 
numerical efficiency, which enlightens the importance of obtaining 
sharp convergence rates exploiting the properties of $f$ and $g$. 
In this context, sharp linear convergence rates can be obtained for 
several splitting algorithms under strong convexity of $f$ and/or $g$
\cite{Taylor18,BoydGiss,Davis17,Yin20,Ryu20}, which can be 
extended when the strong convexity is satisfied on
particular manifolds in the case of partly smooth functions 
\cite{Lewis02,Fadi17}. Moreover, sub-linear convergence rates of 
some {first-order} methods
depending on the {Kurdyka-Lojasiewicz (KL)-exponent} are 
obtained in \cite{AttouchKL} when $f+g$ is a KL-function (see 
\cite{BolteKL}). 
Since KL-exponents 
are usually difficult to compute \cite{PeyKL}, we focus on global strong 
convexity 
assumptions when we aim at finding linear convergence rates.

{The previous discussion devoted to linear convergence rate for 
optimization problems} also holds in the context of monotone 
operators, which appear naturally from primal-dual {first-order} 
optimality conditions 
of optimization problems involving linear operators (see, 
e.g., 
\cite{MS,He_Yuan12,Condat_L_2013_jota_pdhg,Vu_B_2013_j-acm_pdhg}).
We generalize our study of splitting algorithms 
involving implicit and/or 
explicit steps in the context of cocoercive equations. In the presence 
of strong {monotonicity,} we compare linear convergence rates 
of the methods in this context.

\vspace{.3cm}
\noindent \textbf{Contributions} --
In the case when $f$ is strongly convex, we compare 
the Lipschitz-continuous constants of the operators governing 
the gradient method, FBS, PRS, and DRS, which leads to a 
comparison of their linear convergence 
rates. This gives a theoretical support to the results obtained in 
\cite{combettesGlaudin} for the strongly convex case. In the context of 
strongly monotone cocoercive equations,
we provide the linear convergence rates of the four algorithms under 
study, which are larger than the rates in the optimization context.
We also provide an improved 
convergence rate for DRS inspired {by} 
\cite{Giselsson2017,BoydGiss}, 
which exploits the fully smooth context, which is replicated in the 
cocoercive setting. In addition, we obtain a lower convergence 
rate for {the} gradient method in the strongly monotone and 
cocoercive 
setting inspired {by} \cite{Giselsson2017}. 


Based on the obtained 
convergence rates, a second contribution provides efficiency regions of strong 
convexity 
and Lipschitz parameters of $f$ and $g$ identifying the most 
efficient 
algorithm.


A third contribution is to provide several experiments comparing the 
theoretical rates and the numerical behavior of the four methods 
under study in the presence of high and low strong convexity 
parameters.
We obtain that proximal-based schemes PRS and DRS are 
more efficient than EA and FBS in the context of 
piecewise constant denoising and image restoration.

\vspace{.3cm}
\noindent \textbf{Outline} -- In Section~\ref{sec:prelim} we provide the 
results and concepts needed throughout the paper and the
state-of-the-art on convergence properties of the algorithms under 
study. In Section~\ref{sec:coco}, we provide and compare the 
Lipschitz continuous
constants of the operators governing the methods under study
in the cocoercive-strongly monotone setting and our results are 
refined
in Section~\ref{sec:opti} for the particular smooth strongly convex 
optimization context. We also provide efficiency regions 
depending on 
the parameters of the problem identifying the most efficient 
algorithm.
We finish with numerical experiments in 
Section~\ref{sec:exp}.

\section{Preliminaries, problem, and state-of-the-art}
\label{sec:prelim}
{In this section}, we provide our notation, concepts, and results
needed on this paper split in
fixed point theory, 
monotone operator theory, convex analysis, and convergence of 
several algorithms. 
{Throughout this paper,} $\HH$ is a real Hilbert space endowed 
with the inner product $\scal{\cdot}{\cdot}$.
{A sequence $(x_k)_{k\in 
\NN}$ in $\HH$ converges weakly to $x \in \HH$ if,
for every 
$y\in \HH$, $\lim_{k\to +\infty} \langle y\,\vert\, x_k - 
{x}\rangle = 0$, it converges strongly if 
$\lim_{k\to +\infty} \Vert x_k - {x}\Vert = 0$, and it converges 
linearly at rate $\omega\in\left[0,1\right[$
if, for every $k\in\NN$, $\|x_{k}-{x}\|\le \omega^k\|x_{0}-{x}\|$.}

\subsection{Fixed point theory}

An operator $\Phi\colon\HH\to\HH$ is 
$\omega$-Lipschitz 
continuous for some $\omega\in\RP$ if
\begin{equation}
	(\forall x\in\HH)(\forall y\in\HH)\quad 
	\|\Phi x-\Phi y\|\le \omega\|x-y\|,
\end{equation}
and $\Phi$ is nonexpansive if it is $1$-Lipschitz continuous.
The following convergence result, derived from 
\cite[Theorem~1.50]{bauschke2011:Convex_analysis},
is known as the Banach-Picard theorem and asserts 
the strong and linear convergence of iterations generated by 
repeatedly applying a $\omega$-Lipschitz continuous operator
when $\omega\in\left[0,1\right[$ {to a fixed point of $\Phi$, where 
the set of fixed points is $\Fix \Phi = \{x\in \HH \,\vert \, x =  \Phi 
x\}$.} 
\begin{proposition}
\label{p:convlinear}
Let $\omega\in\left[0,1\right[$, let $\Phi\colon\HH\to\HH$ be a 
$\omega$-Lipschitz continuous operator, and let $x_0\in\HH$. 
Set 
\begin{equation}
	(\forall k\in\NN)\quad x_{k+1}=\Phi x_k.
\end{equation}
Then, $\Fix \Phi=\{\widehat{x}\}$ for some $\widehat{x}\in\HH$ and we 
have
\begin{equation}
	(\forall k\in\NN)\quad \|x_{k}-\widehat{x}\|\le 
	\omega^k\|x_{0}-\widehat{x}\|.
\end{equation}
Moreover, $(x_k)_{k\in\NN}$ converges strongly to $\widehat{x}$ with 
linear convergence rate $\omega$.
\end{proposition}



\subsection{Monotone operator theory}
For every set-valued operator $\mathcal{M}\colon\HH\to 2^{\HH}$,
${\rm 
gra}(\mathcal{M})=\menge{(x,u)\in\HH\times\HH}
{u\in\mathcal{M}x}$ is the
graph of $\mathcal{M}$, {$\mathcal{M}^{-1}\colon 
u\mapsto \menge{x\in\HH}{u\in \mathcal{M}x}$ is the inverse of 
$\mathcal{M}$,}
$\mathcal{M}$ is 
monotone if and only if it satisfies, for every $(x,u)$ and $(y,v)$ in 
${\rm gra}(\mathcal{M})$, $\scal{u-v}{x-y}\ge 0$,
and it is maximally monotone if {it is monotone and, for every 
monotone operator 
${\mathcal{T}}\colon\HH\to 2^{\HH}$, ${\rm 
gra}(\mathcal{M})$ is not properly contained in ${\rm 
gra}({\mathcal{T}})$.} {$\Id\colon \HH\to\HH$ 
stands 
for the identity operator.}
For every monotone operator $\mathcal{M}\colon 
\HH\to2^{\HH}$, 
$J_{\mathcal{M}}=(\Id+\mathcal{M})^{-1}$
is the resolvent of $\mathcal{M}$, which is single-valued.
In addition, if $\mathcal{M}$ is maximally monotone, then 
$J_{\mathcal{M}}$ is everywhere defined and 
nonexpansive 
\cite[Proposition~23.8]{bauschke2011:Convex_analysis}.

For every $\eta\in\RP$, we define the class  
$\mathcal{C}_{\eta}$ of $\eta$-cocoercive operators
$\mathcal{M}\colon \HH\to\HH$ satisfying, for every $x$ and $y$ in 
$\HH$,
\begin{equation}
\label{e:coco}
\scal{\mathcal{M}x-\mathcal{M}y}{x-y}
\ge\eta\|\mathcal{M}x-\mathcal{M}y\|^2.
\end{equation}
In particular, $\mathcal{C}_{0}$ is the class of single-valued 
monotone operators.  
Note that, if $\mathcal{M}\in\mathcal{C}_{\eta}$ for some 
$\eta>0$, then $\mathcal{M}$ is maximally monotone in view of
\cite[Corollary~20.28]{bauschke2011:Convex_analysis}.

An operator 
$\mathcal{M}\colon \HH\to\HH$ is $\rho$-strongly monotone for 
some $\rho\in\RPP$
if, for every $x$ and $y$ in $\HH$,
$\scal{\mathcal{M}x-\mathcal{M}y}{x-y}\ge\rho\|x-y\|^2$.

\subsection{Convex analysis}
We denote by $\Gamma_0(\HH)$ the class of functions
$h\colon \HH\to\RX$ which are proper, lower semicontinuous,
and convex. For every $h\in \Gamma_0(\HH)$, the maximally 
monotone operator
\begin{equation}
\partial h\colon x\mapsto \menge{u\in\HH}{(\forall y\in\HH)\: h(x)
+\scal{y-x}{u}\le h(y)}
\end{equation}
is the subdifferential of $h$ and
$\textrm{Argmin}_{x\in\HH}h(x)$ is the set of solutions to 
the problem of minimizing $h$ over $\HH$.
For every $h\in\Gamma_0(\HH)$, it follows from 
\cite[Proposition~17.4]{bauschke2011:Convex_analysis} that
$\widehat{x}\in\textrm{Argmin}_{x\in\HH}h(x)$ if and only if $0\in \partial 
h(\widehat{x})$ and
the proximity operator 
of $h$ is defined by
\begin{equation}
	\label{e:defprox}
	\prox_{h}\colon 
	x\mapsto\arg\min_{y\in\HH}\left(\!h(y)+\frac{1}{2}\|y-x\|^2\!\right),
\end{equation}
which is well defined and single-valued because the objective function 
in~\eqref{e:defprox} is strongly convex. We have 
$\prox_h=J_{\partial h}$ and it reduces to $P_C$,
the projection operator onto a closed convex set $C$, when 
$h=\iota_C$ is the indicator function of $C$, which takes the value 
$0$ in $C$ and $+\infty$ outside.

For every $L\ge 0$, we consider the class 
$\mathscr{C}^{1,1}_L(\HH)$ of 
functions $h\colon\HH\to\RR$ satisfying:
\begin{itemize}
	\item  $h$ is G\^ateaux differentiable in $\HH$, i.e., for every 
	$x\in\HH$ 
	there exists a linear bounded operator $Dh(x)\colon\HH\to\RR$ 
	such 
	that, for every $d\in\HH$,
	\begin{equation}
		Dh(x)d=\lim_{t\downarrow 
			0}\frac{h(x+td)-h(x)}{t}=\scal{\nabla h(x)}{d},
	\end{equation}
	where we denote by $\nabla h(x)\in\HH$ the Riesz-Fr\'echet 
	representant, and
	
	\item $\nabla h\colon \HH\to\HH$ is $L$-Lipschitz continuous. 
\end{itemize}
Observe that, in view of 
\cite[Corollary~17.42]{bauschke2011:Convex_analysis},
every function in $\mathscr{C}^{1,1}_L(\HH)$ is Fr\'echet 
differentiable.
The following proposition is a direct consequence of 
{\cite[Proposition~18.15]
	{bauschke2011:Convex_analysis}} and
asserts that every convex function $h\in 
\mathscr{C}^{1,1}_L(\HH)$ 
satisfies that
$\nabla h$ is $1/L$-cocoercive and vice versa. This result
provides a subclass of $\mathcal{C}_{1/L}$ composed {with} 
gradients 
of convex functions in $\mathscr{C}^{1,1}_L(\HH)$.
\begin{proposition}
	\label{p:BH}
	Let $L\ge 0$ and let $h\colon\HH\to\RR$ be a convex 
	function. 
	Then 
	the following are 
	equivalent:
	\begin{enumerate}
		\item $h\in \mathscr{C}^{1,1}_L(\HH)$.
		\item $h$ is Fr\'echet differentiable and, for every 
		$(x,y)\in\HH^2$,
			$\scal{x-y}{\nabla h(x)-\nabla h(y)}\le L\|x-y\|^2$.
		\item $h$ is Fr\'echet differentiable and $\nabla h\in 
		\mathcal{C}_{1/L}$.
	\end{enumerate}
\end{proposition}
A function $h\in\mathscr{C}_L^{1,1}(\HH)$ 
is $\rho$-strongly convex, for some $\rho\in\RPP$, if 
$h-\frac{\rho}{2} \Vert \cdot \Vert_2^2$ is convex or, equivalently, if 
$\nabla h$ is $\rho$-strongly monotone.

{In Table~\ref{tab:conCAOT}, we summarize the connections 
between convex analysis and operator theory. } For further details 
and properties of monotone operators and 
convex functions in Hilbert spaces, we refer the reader to
\cite{bauschke2011:Convex_analysis}.

\begin{table}[h]
\centering
\begin{tabular}{|c|c|}
\hline
{Convex analysis} & {Operator Theory}\\
\hline
\hline
{$h$ convex, differentiable, and  
$\nabla h$ is $L$-Lipschitz} & {$\nabla 
h$ is $1/L$-cocoercive}\\
{$h$ is $\rho$-strongly convex} 
& {$\partial h$ is $\rho$-strongly monotone}\\
{$h\in \Gamma_0(\HH)$} & 
{$\partial h$ maximally monotone}\\
{$\prox_h$} & {$J_{\partial f}=(\Id 
+ \partial h)^{-1}$}\\
\hline
\end{tabular}
\caption{{Connection between convex analysis and operator 
theory}.\label{tab:conCAOT}}
\end{table}

\subsection{Problem and algorithms}
\label{ssec:probalgo}
{In this paper,} we study several splitting algorithms in the context 
of
the monotone inclusion:
\begin{equation}
\label{e:moninc}
\text{find}\quad x\in\HH\quad\text{such that }\quad 
0\in \mathcal{A}x+\mathcal{B}x,
\end{equation}
where $\mathcal{A}\colon\HH\to 2^{\HH}$ and
$\mathcal{B}\colon\HH\to 2^{\HH}$ are maximally monotone 
operators.
The problem in \eqref{e:moninc} models several problems in game 
theory \cite{Nash}, and optimization problems as considered 
in signal and image processing 
\cite{Combettes_P_2007_inbook,
 Briceno_L_2011_j-math-imaging-vis_pro_ami, 
Cai_JF_2012_j-ams_ima_rtv,
Pustelnik_N_20016_j-w-enc-eee_wav_bid,Fessler_JA_2020_j-ieee-spm},
 among other areas.
In the particular case when $\mathcal{A}=\partial f$ and 
$\mathcal{B}=\partial g$ for some functions
$f$ and $g$ in $\Gamma_0(\HH)$, the 
convex optimization problem (under standard qualification 
conditions) 
	\begin{equation}
\label{e:convopt}
	\minimize{x\in\HH}{f(x)+g(x)},
\end{equation}
is an important particular 
instance of the problem in \eqref{e:moninc}
in view of 
\cite[Proposition~17.4]{bauschke2011:Convex_analysis}. 

In order to solve the problem in \eqref{e:moninc}, the 
algorithms we consider generate recursive 
sequences via Banach-Picard iterations of the form
\begin{equation}
	\label{e:BPiter}
(\forall k\in\NN)\quad 	x_{k+1}=\Phi x_k,
\end{equation}
where $x_0\in\HH$ and $\Phi\colon\HH\to\HH$ is a suitable 
nonexpansive operator which incorporates resolvents and/or 
explicit
computations of $\mathcal{A}$ and $\mathcal{B}$ and such that
 we can recover a 
solution in 
$(\mathcal{A}+\mathcal{B})^{-1}(\{0\})$ from its fixed points. 
More precisely, in this paper we study the following algorithms
for solving the problem in \eqref{e:moninc}.

\vspace{.3cm}
\noindent \textbf{Explicit algorithm (EA)} -- 
It corresponds to apply \eqref{e:BPiter} with the explicit operator
\begin{equation}
	\label{e:defexpl}
	\Phi=G_{\tau(\mathcal{A}+\mathcal{B})}:=\Id-\tau(\mathcal{A}+\mathcal{B}),
\end{equation}
 for some 
$\tau>0$, {leading to the following iterations with $x_0\in 
\mathcal{H}$ and
\begin{equation}
(\forall k\in\NN)\quad 	x_{k+1}= x_k-\tau (\mathcal{A}x_k +  \mathcal{B}x_k).
\end{equation}}
 EA can be seen as an 
explicit Euler 
discretization of the dynamical system governed by 
$\mathcal{A}+\mathcal{B}$ in the single-valued case
\cite[Section~2.4]{PeySor}. In the particular case when 
$\mathcal{A}=\nabla 
f$ and $\mathcal{B}=\nabla 
g$ for smooth convex functions $f$ and $g$, EA 
corresponds to gradient descent \cite{Cauchy,Haskell}.
It is clear that, for every $\tau>0$, 
$(\mathcal{A}+\mathcal{B})^{-1}(0)=\Fix 
G_{\tau(\mathcal{A}+\mathcal{B})}$.

\vspace{.3cm}
\noindent\textbf{Proximal 
Point 	Algorithm (PPA)} -- It is proposed in \cite{Martinet}
for a variational inequality problem and by \cite{Rock76} in the 
maximally monotone context.
This algorithm corresponds to the iteration in \eqref{e:BPiter} 
governed by the resolvent
\begin{equation}
\Phi=J_{\tau(\mathcal{A}+\mathcal{B})}=(\Id+\tau(\mathcal{A}+\mathcal{B}))^{-1}.
\end{equation}
for some $\tau>0$, {leading to the following iterations with $x_0\in 
\mathcal{H}$ and
\begin{equation}
(\forall k\in\NN)\quad 	x_{k+1}= J_{\tau(\mathcal{A}+\mathcal{B})}x_k.
\end{equation}}
PPA can be seen as an implicit 
discretization of the dynamical system governed by 
$\mathcal{A}+\mathcal{B}$
\cite[Section~2.3]{PeySor}. In the particular case when 
$\mathcal{A}=\partial 
g$ and $\mathcal{B}=\partial f$ for some convex functions $f$ 
and $g$ satisfying standard qualification conditions, $J_{\tau 
	(\mathcal{A}+\mathcal{B})}=\prox_{\tau(f+g)}$ is the proximity 
	operator defined in \eqref{e:defprox} and motivates the name 
	to 
	the 
	algorithm.
Each (implicit) step of PPA
includes the resolution of a non-linear equation, but, in a large class of 
operators, this equation has an explicit solution or it is easy to solve.
It is clear that, for every $\tau>0$, 
$(\mathcal{A}+\mathcal{B})^{-1}(0)=\Fix 
J_{\tau(\mathcal{A}+\mathcal{B})}$.

\vspace{.3cm}
\noindent \textbf{Forward-Backward splitting (FBS)} -- 
It follows 
from \eqref{e:BPiter} with the Forward-Backward operator
\begin{equation}
	\label{e:defFB}
	\Phi=T_{\tau\mathcal{B},\tau\mathcal{A}}=J_{\tau\mathcal{B}}\circ 
	G_{\tau\mathcal{A}}=(\Id+\tau\mathcal{B})^{-1}(\Id-\tau\mathcal{A}),
\end{equation}
for some $\tau>0$, {leading to {the following iterations with 
$x_0\in \mathcal{H}$ and}
\begin{equation}
(\forall k\in\NN)\quad 	x_{k+1}= J_{\tau\mathcal{B}}(x_k-\tau\mathcal{A} x_k).
\end{equation}}
which
alternates {between} explicit and implicit steps. In the case when 
$\mathcal{A}=\nabla 
f$ and $\mathcal{B}=\partial g$, for some $f$ 
and $g$ in $\Gamma_0(\HH)$,  $J_{\tau\mathcal{B}} = \prox_{\tau g}$ for every $\tau>0$ and FBS is the 
proximal gradient algorithm (see, e.g., \cite{sims}). This method finds its 
roots in the 
projected gradient method \cite{levitin1966constrained} (case 
$g=\iota_C$ for some closed convex set $C$). In the 
context of variational inequalities appearing in some PDE's, a 
generalization of the projected gradient method is proposed in 
\cite{BreSib, Mercier, Sibony}.
It follows from
\cite[Proposition~26.1(iv)(a)]{bauschke2011:Convex_analysis} that,
for every $\tau>0$, $(\mathcal{A}+\mathcal{B})^{-1}(0)=\Fix 
	T_{\tau\mathcal{B},\tau \mathcal{A}}$.

\vspace{.3cm}
\noindent \textbf{Peaceman-Rachford splitting (PRS)} -- 
 This 
scheme fo\-llows 
from \eqref{e:BPiter} with the Peaceman-Rachford operator 
\begin{equation}
	\label{e:defPR}
	\Phi=R_{\tau\mathcal{B},\tau\mathcal{A}}
=(2J_{\tau\mathcal{B}}-\Id)\circ 
	(2J_{\tau\mathcal A}-\Id),
\end{equation}
for some $\tau>0$, {leading to the following iterations with $x_0\in 
\mathcal{H}$ and
\begin{equation}
(\forall k\in\NN)\quad 	\begin{cases}
x_{k+1/2}= J_{\tau\mathcal A}x_k,\\
x_{k+1} = 2J_{\tau\mathcal{B}}(2 x_{k+1/2} - x_{k}) - 2 x_{k+1/2} + x_{k}.
\end{cases}
\end{equation}} 
PRS is first proposed in \cite{PRS}
for solving some linear systems derived 
from discretizations of PDE's and it is studied in the 
non-linear monotone case in \cite{LionsMercier}.
It follows from 
\cite[Proposition~26.1(iii)(b)]{bauschke2011:Convex_analysis}
that, for every $\tau>0$, $	(\mathcal{A}+\mathcal{B})^{-1}(0)
	=J_{\tau\mathcal{A}}(\Fix 
	R_{\tau\mathcal{B},\tau \mathcal{A}})$.
As before, we recover PRS in the optimization context by 
using the identity $J_{\partial h}=\prox_h$ for 
$h\in\Gamma_0(\HH)$. 

\vspace{.3cm}
\noindent \textbf{Douglas-Rachford splitting (DRS)} -- 
This 
scheme fo\-llows 
from \eqref{e:BPiter} with Douglas-Rachford operator
\begin{equation}
	\label{e:defDR}
	\Phi=S_{\tau\mathcal{B},\tau\mathcal{A}} =
	\frac{\Id+R_{\tau\mathcal{B},\tau\mathcal{A}}}{2}\\
	=J_{\tau\mathcal{B}}(2J_{\tau\mathcal{A}}-\Id)+\Id- J_{\tau\mathcal{A}},
\end{equation}
for some $\tau>0$, which is the average 
between $\Id$ and $R_{\tau\mathcal{B},\tau\mathcal{A}}$, 
{leading to the following iterations with $x_0\in \mathcal{H}$ and
\begin{equation}
(\forall k\in\NN)\quad 	\begin{cases}
x_{k+1/2}= J_{\tau\mathcal A}x_k,\\
x_{k+1} = J_{\tau\mathcal{B}}(2 x_{k+1/2} - x_{k}) - x_{k+1/2} + x_{k}.
\end{cases}
\end{equation}} 
The 
algorithm is first proposed for solving some linear systems derived 
from discretizations of PDE's \cite{DR} and it is studied in the 
non-linear monotone case in \cite{LionsMercier}. It follows from 
\cite[Proposition~26.1(iii)(b)]{bauschke2011:Convex_analysis}
that, for every $\tau>0$, $	
(\mathcal{A}+\mathcal{B})^{-1}(0)=J_{\tau\mathcal{A}}(\Fix 
	S_{\tau\mathcal{B},\tau \mathcal{A}})$.
As before, we recover DRS in the optimization context by 
using the identity $J_{\partial h}=\prox_h$ for 
$h\in\Gamma_0(\HH)$.

\subsection{State-of-the-art on convergence of algorithms}
If $\mathcal{M}=\mathcal{A}+\mathcal{B}$ is 
strongly monotone{, $\eta$-cocoercive,} and 
$\tau\in\left]0,2\eta\right[$,
$G_{\tau\mathcal{M}}$ is Lipschitz continuous with 
constant in $\left]0,1\right[$ \cite[Fact~7]{Yin20} and EA 
achieves linear convergence in view of 
Proposition~\ref{p:convlinear}. 
In addition, $J_{\tau\mathcal{M}}$ is Lipschitz continuous with 
constant in $\left]0,1\right[$ and PPA 
converges linearly
\cite[Proposition~23.13]{bauschke2011:Convex_analysis}. 
However, when 
$\mathcal{M}=\mathcal{A}+\mathcal{B}$, the computation 
of $J_{\tau\mathcal{M}}$ can be difficult, and other splitting methods 
as EA, FBS, PRS, and DRS can be considered in order to reduce the 
computational time by iteration.

If we assume the strong monotonicity of 
$\mathcal{A}$ or $\mathcal{B}$,
the linear convergence of FBS is guaranteed 
\cite[Theorem~26.16]{bauschke2011:Convex_analysis}, which 
follows from the Lipschitz continuity of 
$T_{\tau\mathcal{B},\tau\mathcal{A}}$
with Lipschitz constant in $\left]0,1\right[$. 
In \cite{ChenRock} the authors provide a detailed analysis of the 
convergence rates 
of FBS in the strongly monotone context. 
If $\mathcal{A}$ is not 
cocoercive the convergence of FBS 
is not guaranteed and, if it is not single-valued,  it is not applicable. In 
these contexts PRS and DRS can be 
used if $J_\mathcal{A}$ is not difficult to compute. In the case 
when $\mathcal{A}$ and $\mathcal{B}$ are merely maximally 
monotone, reflections $2J_{\mathcal{A}}-\Id$ and 
$2J_{\mathcal{B}}-\Id$ are merely nonexpansive, and the 
convergence of PRS is not guaranteed. This motivates the 
average with $\Id$ in \eqref{e:defDR}, which allows to obtain the  
weak convergence of DRS to a solution. Under the cocoercivity 
assumption on $\mathcal{A}$, the weak 
convergence of PRS is guaranteed in 
\cite[Corollary~1\& Remark~2(2)]{LionsMercier}.
If in addition we suppose the strong monotonicity 
of $\mathcal{A}$,
the reflection $2J_{\mathcal{A}}-\Id$ is Lipschitz continuous with 
constant in $\left]0,1\right[$ \cite{Giselsson2017} and, therefore, 
PRS converges linearly to a solution. This property 
also holds for DRS, but with a larger convergence rate. 
Of course, previous properties are inherited by the algorithms in 
the particular optimization context, sometimes with better 
convergence rates by exploiting the variational formulation 
\cite{Yin20,Davis17,Taylor18,BoydGiss}.  

In summary, without any cocoercivity on {the} problem 
\eqref{e:moninc} the only available convergent method 
is DRS, if resolvents are easy to compute. However, in the fully 
cocoercive setting all the methods under study are convergent 
and can be implemented, and there is no theoretical/numerical 
comparison of these methods in the literature in this context. 
In this paper, as stated in Section~\ref{sec:intro}, we focus on 
cocoercive equations involving the sum of two operators, in which one 
of them is strongly monotone. Even if restrictive, this setting allows us 
to provide and compare the optimal linear convergence rates of the 
four algorithms described above. This analysis is further refined for 
minimization problems involving the 
sum of two smooth convex functions with Lipschitzian gradients, in 
which one of them is strongly convex. We also indicate which 
algorithm is more efficient depending on strong convexity and 
Lipschitz parameters.
We start by studying cocoercive equations.
\section{Cocoercive equations}
\label{sec:coco}
In this section we study properties of different numerical schemes 
for solving the following cocoercive equation.
\begin{problem}
	\label{prob:mainsplit}
	Let $(\alpha,\beta)\in\RPP^2$, let $\rho\in\left]0,\alpha^{-1}\right]$, 
	let
	$\mathcal{A}\in\mathcal{C}_{\alpha}$ be $\rho$-strongly monotone, 
	and let
	$\mathcal{B}\in\mathcal{C}_{\beta}$. 
	The problem is to 
	\begin{equation}
		\text{find}\quad x\in\HH\quad\text{such that }\quad 
		\mathcal{A}x+\mathcal{B}x=0,
	\end{equation}
	under the assumption that solutions exist.
\end{problem}
{Note that, any $\rho$-strongly convex and 
$\alpha$-cocoercive operator $\mathcal{A}$ should satisfy 
$\rho\le 1/\alpha$, 
since, for every $x$ and $y$ in $\HH$, we have
\begin{equation}
\rho\|x-y\|\le \scal{x-y}{\mathcal{A}x-\mathcal{A}y}\le 
\|x-y\|\|\mathcal{A}x-\mathcal{A}y\|\le \alpha^{-1}\|x-y\|^2.
\end{equation}
Therefore, the assumption $\rho\in\left]0,\alpha^{-1}\right]$ is 
{not 
restrictive}. 
 In order to motivate the cocoercive equation in 
Problem~\ref{prob:mainsplit}, {we} consider the following 
example.}
\begin{example}
\label{ex:coco1}
{Set $\HH=\RR^N$, let $A$ and $D$ be $M\times N$ and 
$K\times N$ real matrices, respectively.  
Let $z\in\RR^M$, let $h\in\mathscr{C}_{L}^{1,1}(\RR^K)$, and
consider the problem 
\begin{equation}
\label{e:excoco}
\min_{x\in \RR^n}\frac{1}{2}\|Ax-z\|^2+h(Dx).
\end{equation}
This minimization problem is typically encountered in image processing when the matrix $D$ represents the discrete gradient and $h$ is 
a smooth version of the $\ell_{1,2}$-norm leading to hyperbolic total-variation \cite{Charbonnier_P_1997_j-ieee-tip_det,Denneulin_L_2021_AA_rha}.}

{
Note that, if $A$ is a full rank matrix and by setting $\mu$ to be the smallest eigenvalue of 
$A^{\top}A$, we have  
$\mu>0$ and the function 
$f\colon x\mapsto\|Ax-z\|^2/2$ is in 
$\mathscr{C}^{1,1}_{\|A\|^2}(\RR^n)$ and it is
$\mu$-strongly convex. 
Moreover, the optimality conditions of a solution 
$\widehat{x}$ to \eqref{e:excoco} read
\begin{equation}
\label{e:copti}
0=\nabla f(\widehat{x})+D^{\top}\nabla h(D\widehat{x}). 
\end{equation}
By defining $\widehat{u}=\nabla h(D\widehat{x})$, it follows 
from \cite[Proposition~16.10]{bauschke2011:Convex_analysis} 
that 
$D\widehat{x}\in \partial h^*(\widehat{u})$ and
\eqref{e:copti} is equivalent to 
\begin{equation}
\label{e:pdopti}
\begin{cases}
0=\nabla f(\widehat{x})+D^{\top}\widehat{u}\\
0\in\partial h^*(\widehat{u})-D\widehat{x}.
\end{cases}
\end{equation}
Therefore, 
by defining 
\begin{equation}
\label{e:maxmonops}
\begin{cases}
\mathcal{A}\colon (x,u)\mapsto \{\nabla 
f({x})-\eta x\}\times(\partial h^*({u})-\eta u)\\
\mathcal{B}\colon (x,u)\mapsto (\eta x+D^{\top}u,\eta u-Dx),
\end{cases}
\end{equation}
where $\eta\in\left]0,\min\{\mu,\frac{1}{L}\}\right[$,
\eqref{e:pdopti} is equivalent to find 
$(\widehat{x},\widehat{u})\in\RR^n\times\RR^p$ such that 
\begin{equation}
\label{eq:monincexample}(0,0)\in\mathcal{A}(\widehat{x},\widehat{u})
+\mathcal{B}(\widehat{x},\widehat{u}),\end{equation} 
where $\mathcal{A}$ is $(\min\{\mu,\frac{1}{L}\}-\eta)$-strongly monotone, and
$\mathcal{B}$ is $\eta/\|\mathcal{B}\|$-cocoercive. 
Hence \eqref{eq:monincexample} is a special instance of \eqref{e:moninc}.
If we include the assumption that  $h$ is a $\rho$-strongly convex 
function, 
$\mathcal{A}$ is  
$(\min\{\rho,\|A\|^{-2}\}/(1+\eta\max\{(1/\mu,L\})^2)$-cocoercive, 
and 
\eqref{e:pdopti} becomes a particular instance of 
Problem~\ref{prob:mainsplit}. The proof of the properties on 
$\mathcal{A}$ and $\mathcal{B}$ above are detailed in the 
appendix (see Section~\ref{sec:cocomon}).}
\end{example}

It turns out that, because of the strong monotonicity assumption, 
there exists a unique solution
{$\widehat{x}\in(\mathcal{A+B})^{-1}(0)$} and the operators 
 $G_{\tau(\mathcal{A}+\mathcal{B})}$,
$T_{\tau\mathcal{B},\tau\mathcal{A}}$, 
$R_{\tau\mathcal{B},\tau\mathcal{A}}$, 
and 
$S_{\tau\mathcal{B},\tau\mathcal{A}}$ defined in 
\eqref{e:defexpl}--\eqref{e:defDR}
are $\omega(\tau)$-Lipschitz continuous for some 
 $\omega(\tau)\in \left]0,1\right[$, under suitable conditions on $\tau$.
The Lipschitz continuous constant of each algorithm corresponds 
to its linear 
convergence rate in view of Proposition~\ref{p:convlinear}, which 
allows the user to compare not only numerically 
but also theoretically the convergence behavior of each method. 
In the next 
proposition, we summarize the convergence rates for the
schemes governed by the operators defined in 
\eqref{e:defexpl}--\eqref{e:defDR} aiming to solve 
Problem~\ref{prob:mainsplit}. 

\begin{proposition}
\label{p:stronglysplit1}
Let $\tau>0$. In the context of Problem~\ref{prob:mainsplit},  the 
following hold:
\begin{enumerate}
\item 
\label{p:stronglysplit1i}
Suppose that 
$\tau\in\left]0,2\beta\alpha/(\beta+\alpha)\right[$. Then
$G_{\tau(\mathcal{A}+\mathcal{B})}$ is 
$\omega_{G}(\tau)$-Lipschitz 
continuous, where
\begin{equation}
\label{e:constantG2strong}
\hspace{-0.5cm}\omega_G(\tau):=
\sqrt{1-\frac{2\tau\rho}{\alpha(2\beta-\tau)}
\big(2\beta\alpha-\tau(\beta+\alpha)\big)}\in\left]0,1\right[.
\end{equation}
In particular, the minimum in \eqref{e:constantG2strong} is 
achieved at
\begin{equation}
	\label{eq:taugrad2}
	\tau^*=\frac{2\beta\alpha}{\sqrt{\beta+\alpha}
		(\sqrt{\beta+\alpha}+\sqrt{\beta})}\quad\text{and}\quad
	 \omega_{G}(\tau^*)
	=\sqrt{1-\frac{4\rho\beta\alpha }{	
	(\sqrt{\beta+\alpha}+\sqrt{\beta})^2}}.
\end{equation}
\item 
\label{p:stronglysplit1ii}
Suppose that $\tau\in\left]0,2\alpha\right[$. Then
$T_{\tau\mathcal{B},\tau\mathcal{A}}$ is 
$\omega_{T_1}(\tau)$-Lipschitz 
continuous, where
\begin{equation}
\label{e:constantFB1strong}
\omega_{T_1}(\tau):=\sqrt{1-\frac{\tau\rho}{\alpha}
(2\alpha-\tau)}\in\left]0,1\right[.
\end{equation}
In particular, the minimum in \eqref{e:constantFB1strong} is 
achieved at
\begin{equation}
	\label{eq:tauFB1}
	\tau^*=\alpha\quad
	\text{ and }\quad \omega_{T_1}(\tau^*)
	=\sqrt{1-\alpha\rho}.
\end{equation}
\item 
\label{p:stronglysplit1iiv}
Suppose that $\tau\in\left]0,2\beta\right]$. Then
$T_{\tau\mathcal{A},\tau\mathcal{B}}$ is 
$\omega_{T_2}(\tau)$-Lipschitz 
continuous, where
\begin{equation}
\label{e:FB2strong}
\omega_{T_2}(\tau):=\dfrac{1}{1+\tau\rho}\in\left]0,1\right[.
\end{equation}
In particular, the minimum in \eqref{e:FB2strong} is 
achieved at
\begin{equation}
	\label{eq:tauFB2}
	\tau^*=2\beta\quad
	\text{ and }\quad \omega_{T_2}(\tau^*)
	=\frac{1}{1+2\beta\rho}.
\end{equation}
\item
\label{p:stronglysplit1iii}
$R_{\tau\mathcal{B},\tau\mathcal{A}}$ and 
$R_{\tau\mathcal{A},\tau\mathcal{B}}$ are 
$\omega_R(\tau)$-Lipschitz 
continuous, where
\begin{equation}
\label{e:PRstrong}
\omega_R(\tau)=\sqrt{\dfrac{\alpha-2\tau\rho\alpha+\tau^2\rho}
{\alpha+2\tau\rho\alpha+\tau^2\rho}}
\in\left]0,1\right[.
\end{equation}
In particular, the minimum in \eqref{e:PRstrong} is 
achieved at
\begin{equation}
	\label{eq:taupr}
	\tau^*=\sqrt{\frac{\alpha}{\rho}}\quad
	\text{ and }\quad \omega_{R}(\tau^*)
	=\sqrt{\dfrac{1-\sqrt{\alpha\rho}}{1+\sqrt{\alpha\rho}}}.
\end{equation}
\item
\label{p:stronglysplit1iv}
$S_{\tau\mathcal{B},\tau\mathcal{A}}$ and 
$S_{\tau\mathcal{A},\tau\mathcal{B}}$ are $\omega_S(\tau)$-Lipschitz 
continuous, where
\begin{equation}
	\label{e:constDRgen}
\omega_S(\tau)=\min\left\{\frac{1+\omega_R(\tau)}{2},
\frac{\beta+\tau^2\rho}{\beta+\tau\beta\rho+\tau^2\rho}\right\}
\in\left]0,1\right[.
\end{equation}
In particular, the minimum in \eqref{e:constDRgen} is 
achieved at
\begin{equation}
	\label{eq:taudrgen}
	\tau^*=
\begin{cases}
\sqrt{\frac{\alpha}{\rho}},\quad&\text{if}\:\:\beta\leq 
\frac{4\alpha}{(1+\sqrt{1-\alpha\rho})^2};\\
\sqrt{\frac{\beta}{\rho}},\quad&\text{otherwise},
\end{cases}\quad\text{and}\quad 
 \omega_{S}(\tau^*)=
\begin{cases}
\frac{1+\sqrt{1-{\alpha\rho}}}{1+\sqrt{1-{\alpha\rho}}+\sqrt{\alpha\rho}}
,\quad&\text{if}\:\:\beta\leq 
\frac{4\alpha}{(1+\sqrt{1-\alpha\rho})^2};\\
\frac{2}{2+\sqrt{\beta\rho}},\quad&\text{otherwise}.
\end{cases}
\end{equation}
\end{enumerate}
\end{proposition}
\noindent The proof is provided in 
Appendix~\ref{app:stronglysplit1}.
Observe that 
Proposition~\ref{p:stronglysplit1}\eqref{p:stronglysplit1i}
is a new result, in which the Lipschitz constant of the explicit 
operator is improved with respect to 
considering a single operator when {splitting} 
is possible (see 
Remark~\ref{rem:compsplitornot}).
Proposition~\ref{p:stronglysplit1}\eqref{p:stronglysplit1ii} provides a 
smaller
Lipschitz-constant for operator $T_{\tau\mathcal{B},\tau\mathcal{A}}$ 
than in \cite[Remarque~3.1(2)]{Mercier80}, 
\cite[Theorem~2.4]{ChenRock},
\cite[Proposition~1(d)]{Tseng91}, and
\cite[Proposition~26.16(ii)]{bauschke2011:Convex_analysis}, 
 by exploiting the cocoercivity of $\mathcal{A}$. On the other hand,
in Proposition~\ref{p:stronglysplit1}\eqref{p:stronglysplit1iiv} we 
obtain a better Lipschitz constant for 
$T_{\tau\mathcal{A},\tau\mathcal{B}}$ than in 
\cite[Proposition~1(d)]{Tseng91} and \cite[Theorem~2.4]{ChenRock}, 
and we
recover the Lipschitz constant  in 
\cite[Proposition~26.16(i)]{bauschke2011:Convex_analysis}, 
but we obtain a smaller Lipschitz constant by allowing 
$\tau=2\beta$. The Lipschitz constant of 
$R_{\tau\mathcal{A},\tau\mathcal{B}}$ and 
$R_{\tau\mathcal{B},\tau\mathcal{A}}$ in \eqref{e:PRstrong} is 
obtained in 
\cite[Theorem~7.4]{Giselsson2017}, and it is smaller than 
Lipschitz constants in \cite[Theorem~6.5 \& 
Theorem~5.6]{Giselsson2017} which 
are also valid in our context. The constant in 
\eqref{e:constDRgen} is provided in 
\cite[Theorem~7.4]{Giselsson2017} and it is tighter than
the constant obtained in \cite[Proposition~4]{LionsMercier}, 
which does 
not take advantage of the full cocoercivity of the problem.
The Lipschitz constant of 
$S_{\tau\mathcal{A},\tau\mathcal{B}}$ and 
$S_{\tau\mathcal{B},\tau\mathcal{A}}$ in \eqref{e:constDRgen}
is obtained from \cite[Theorem~5.6\,\&\,Theorem~7.4]{Giselsson2017}
by exploiting the cocoercivity of $\mathcal{A}$ and $\mathcal{B}$.
When $\alpha$ is large with respect to $\beta$, our constant is 
sharper than the constant 
in \cite[Corollary~4.2]{Ryu20} (see Figure~\ref{Fig:ryu}), which is 
obtained via computer-assisted analysis. This is because the 
cocoercivity 
of $\mathcal{A}$ is not considered in \cite{Ryu20}.
\begin{figure}
\begin{center}
\includegraphics[scale=0.4]{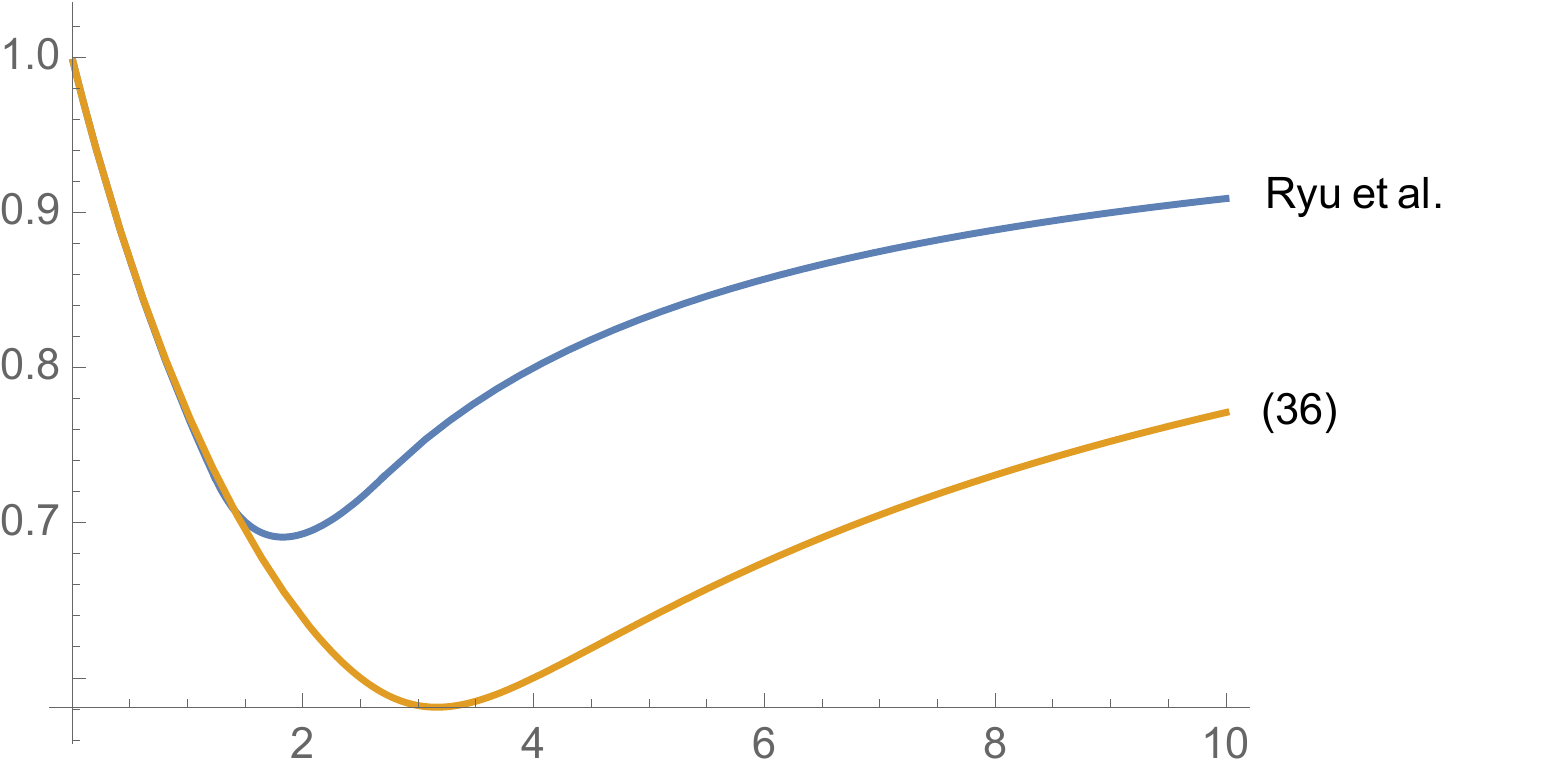}\\
{Rate vs step-size $\tau$}
\end{center}
\caption{\small Comparison between the Lipschitz constants in 
\cite{Ryu20}
and \eqref{e:constDRgen} for DRS when $\beta=1$, $\rho=0.3$, and 
$\alpha=3$.}
\label{Fig:ryu}
\end{figure}

In the case when $\mathcal{B}=0$, 
by taking $\beta\to+\infty$ in parts \ref{p:stronglysplit1i} (or 
\ref{p:stronglysplit1ii}) and \ref{p:stronglysplit1iiv} of 
 Proposition~\ref{p:stronglysplit1} we obtain 
as a direct consequence the following result for EA and PPA in 
the strongly monotone case. The Lipschitz continuous constant
of  EA obtained in \cite[Fact~7]{Yin20} with a geometric proof is 
complemented with analytic arguments in the proof of 
Proposition~\ref{p:stronglysplit1}. The constant of PPA is
proved in 
\cite[Proposition~23.13]{bauschke2011:Convex_analysis}.
\begin{proposition}
\label{p:3}
Suppose that $\mathcal{A}\in\mathcal{C}_{\alpha}$ is $\rho$-strongly 
monotone, for some $\alpha\in\RPP$ and
$\rho\in\left]0,\alpha^{-1}\right]$. Then the following hold.
\begin{enumerate}
\item \label{p:3ii} For every
		$\tau\in\left]0,2\alpha\right[$, $G_{\tau 
		\mathcal{A}}$ is $\omega_{G_0}(\tau)$-Lipschitz 
		continuous, where
\begin{equation}
\label{e:defG0}
\omega_{G_0}:=\sqrt{1-\frac{\tau\rho}{\alpha}
(2\alpha-\tau)}\in\left]0,1\right[.
\end{equation}

\item\label{p:3i} 
For every $\tau >0$,
$J_{\tau\mathcal{A}}$ is 
$\omega_J(\tau)$-Lipschitz continuous, 
where 
\begin{equation}
\omega_J(\tau):=\dfrac{1}{1+\tau\rho}\in\left]0,1\right[.
\end{equation}
	\end{enumerate}
\end{proposition}

\begin{remark}
	\label{rem:compsplitornot}
Observe that $\mathcal{A}+\mathcal{B}$
is $\beta\alpha/(\beta+\alpha)$-cocoercive 
\cite[Proposition~4.12]{bauschke2011:Convex_analysis} and 
$\rho$-strongly monotone.
Moreover, for every $\tau\in\left]0,2\beta\alpha/(\beta+\alpha)\right[$ 
we have
\begin{equation}
\dfrac{\tau\rho}{\beta\alpha}
\big(2\beta\alpha-\tau(\beta+\alpha)\big)< 
\frac{2\tau\rho}{\alpha(2\beta-\tau)}
\big(2\beta\alpha-\tau(\beta+\alpha)\big).
\end{equation}
Therefore $\omega_{G}$ defined in 
\eqref{e:constantG2strong} 
is strictly lower than $\omega_{G_0}$ in \eqref{e:defG0}. 
Moreover, in the case when $\mathcal{B}=0$ 
($\beta\to \infty$), both functions coincide. This new 
result implies that the gradient operator takes advantage of 
the splitting when a part of the monotone inclusion is strongly 
monotone.
\end{remark}

\begin{remark}
In the absence of strong monotonicity ($\rho=0$), EA, FBS, PRS, 
and 
DRS generate weakly convergent sequences. Indeed, even if 
the 
associated fixed point operators are no longer strict contractions, 
Proposition~\ref{p:2} in the appendix asserts that they are 
averaged 
nonexpansive operators, i.e., there exists $\mu\in\left]0,1\right[$ 
such 
that
\begin{equation}
(\forall x\in\HH)(\forall y\in\HH)\quad 
\|\Phi x-\Phi y\|^2\le	\|x-y\|^2-
\left(\dfrac{1-\mu}{\mu}\right)\|(\Id-\Phi)x-(\Id-\Phi)y\|^2.
\end{equation}
The weak convergence hence follows from 
\cite[Proposition~5.16]{bauschke2011:Convex_analysis}.
Note that, in the cocoercive context, the averaged nonexpansive 
property for the fixed point operator associated {with} PRS is a 
new 
result.
\end{remark}

\begin{figure*}[t]
\centering
	\begin{tabular}{cc}
		\includegraphics[width = 6cm]{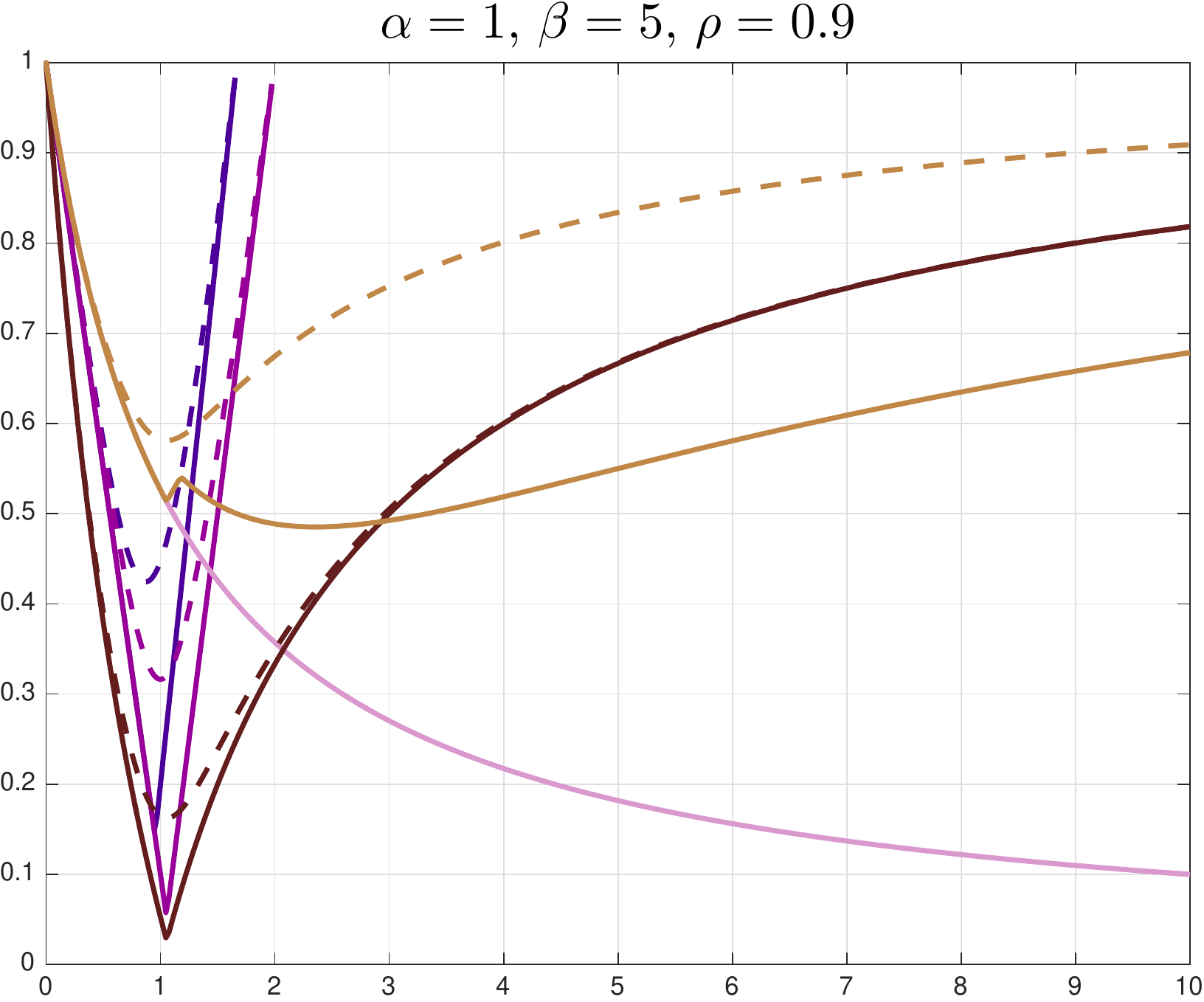} 
		& \includegraphics[width = 
		6cm]{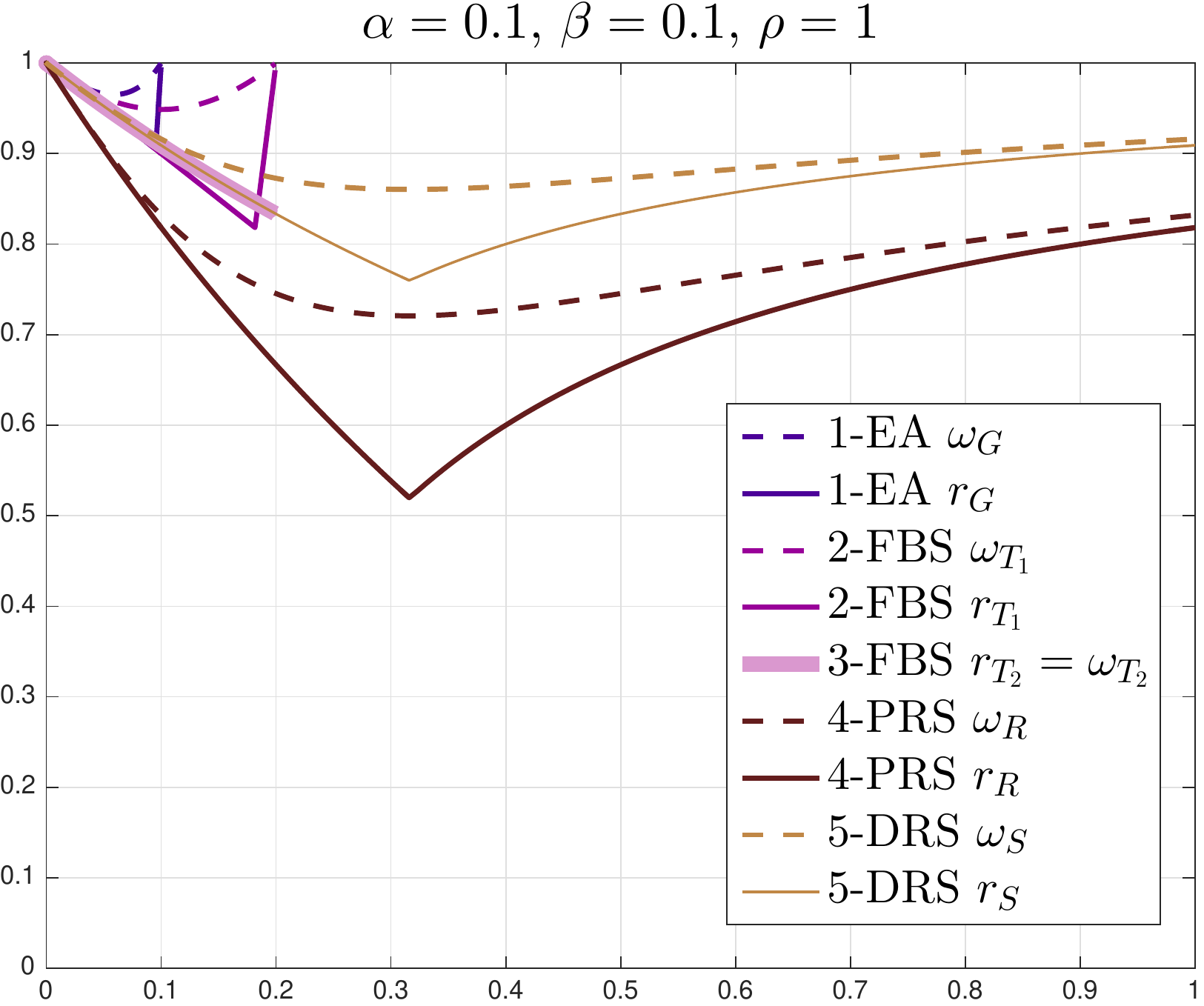}\\
		{(a) Rate vs step-size $\tau$} & {(b) Rate vs step-size $\tau$}
	\end{tabular}
	\caption{{\small 
		Comparison of the convergence rates of EA, FBS, PRS, DRS 
		obtained in Proposition~\ref{p:stronglysplit1opt} (continuous 
		lines) and 
		Proposition~\ref{p:stronglysplit1} (dashed 
		lines) for two choices of $\alpha$, $\beta$, and 
		$\rho$. Note that optimization rates are better than cocoercive 
		rates in general.\label{fig:ratecocoopt}}}
\end{figure*}

\section{Smooth convex optimization}
\label{sec:opti}
In this section we restrict our attention to the following particular 
instance of Problem~\ref{prob:mainsplit}.
\begin{problem}
	\label{prob:mainsplitopt}
Let $(\alpha,\beta)\in\RPP^2$, let $\rho\in\left]0,\alpha^{-1}\right]$, 
let $f\in\mathscr{C}_{1/\alpha}^{1,1}(\HH)$ be $\rho$-strongly convex, 
and let
	$g\in\mathscr{C}_{1/\beta}^{1,1}(\HH)$.
	The problem is to 
	\begin{equation}
		\minimize{x\in\HH}{f(x)+g(x)},
	\end{equation}
	under the assumption that solutions exist.
\end{problem} 

In the context of Problem~\ref{prob:mainsplitopt}, there exists a 
unique solution to 
Problem~\ref{prob:mainsplitopt}, which is denoted by $\widehat{x}$.
Since $\mathcal{A}=\nabla f$ is cocoercive and strongly monotone,
Proposition~\ref{p:stronglysplit1} provides Lipschitz constants
of the operators governing the numerical schemes under study. 
However, in the 
optimization setting the Lipschitz constants can be improved, as 
Proposition~\ref{p:stronglysplit1opt} below asserts. Next, we 
compare the convergence rates and we provide regions depending on 
the parameters $\alpha$, $\beta$, and $\rho$ defining the most 
efficient algorithm in {the} worst-case scenario.
\subsection{Linear convergence rates}
The following result is a refinement of 
Proposition~\ref{p:stronglysplit1},
in which the Lipschitz constants are improved by using the convex 
optimization structure of the problem.  {All the linear convergence 
rates, optimal step-sizes, and associated optimal rates are 
summarized in Table~\ref{tab:sumrates}.}

\begin{proposition}
\label{p:stronglysplit1opt}
Let $\tau>0$. In the context of Problem~\ref{prob:mainsplitopt}, the 
following hold:
\begin{enumerate}
\item 
\label{p:stronglysplit1opti}
Suppose that 
$\tau\in\left]0,2\beta\alpha/(\beta+\alpha)\right[$. Then,
$G_{\tau(\nabla g+\nabla f)}$ is 
${r}_{G}(\tau)$-Lipschitz 
continuous, where
\begin{equation}
\label{e:constantG2strongopt}
{r}_{G}(\tau):=
\max\big\{|1-\tau\rho|,|1-\tau(\beta^{-1}+\alpha^{-1})|\big\}\in\left]0,1\right[.
\end{equation}
In particular, the minimum in \eqref{e:constantG2strongopt} is 
achieved at
\begin{equation}
\label{eq:taugrad}
\tau^*=\dfrac{2}{\rho+\alpha^{-1}+\beta^{-1}}\quad\text{and}\quad
{r}_{G}(\tau^*)
=\dfrac{\alpha^{-1}+\beta^{-1}-\rho}
{\alpha^{-1}+\beta^{-1}+\rho}.
\end{equation}
\item 
\label{p:stronglysplit1optii}
Suppose that $\tau\in\left]0,2\alpha\right[$. Then
$T_{\tau\nabla g,\tau\nabla f}$ is $r_{T_1}(\tau)$-Lipschitz 
continuous, where
\begin{equation}
\label{e:constgradreduced}
r_{T_1}(\tau):=\max\big\{|1-\tau\rho|,|1-\tau\alpha^{-1}|\big\}\in\left]0,1\right[.
\end{equation}
In particular, the minimum in \eqref{e:constgradreduced} is 
achieved at
\begin{equation}
\label{eq:taufb2}
\tau^*=\dfrac{2}{\rho+\alpha^{-1}}\quad
\text{ and }\quad {r}_{T_1}(\tau^*)
=\dfrac{\alpha^{-1}-\rho}
{\alpha^{-1}+\rho}.
\end{equation}
\item 
\label{p:stronglysplit1optiiv}
Suppose that $\tau\in\left]0,2\beta\right]$. Then
$T_{\tau\nabla f,\tau\nabla g}$ is ${r}_{T_2}(\tau)$-Lipschitz 
continuous, where
\begin{equation}
\label{e:consFB2}
{r}_{T_2}(\tau):=\dfrac{1}{1+\tau\rho}\in\left]0,1\right[.
\end{equation}
In particular,  the minimum in \eqref{e:consFB2} is 
achieved at
\begin{equation}
\label{eq:taufb3}
\tau^*=2\beta\quad
\text{ and }\quad {r}_{T_2}(\tau^*)
=\dfrac{1}{1+2\beta\rho}.
\end{equation}
\item
\label{p:stronglysplit1optiii}
$R_{\tau\nabla g,\tau\nabla f}$ and 
$R_{\tau\nabla f,\tau\nabla g}$ are $r_R(\tau)$-Lipschitz 
continuous, where 
\begin{equation}
\label{e:constPR}
r_R(\tau)=\max\left\{\dfrac{1-\tau\rho}{1+\tau\rho},
\dfrac{\tau\alpha^{-1}-1}{\tau\alpha^{-1}+1}\right\}
\in\left]0,1\right[.
\end{equation}
In particular, the minimum in \eqref{e:constPR} is 
achieved at
\begin{equation}
\label{eq:taupropti}
\tau^*=\sqrt{\frac{\alpha}{\rho}}\quad
\text{ and }\quad {r}_{R}(\tau^*)
={\dfrac{1-\sqrt{\alpha\rho}}{1+\sqrt{\alpha\rho}}}.
\end{equation}
\item
\label{p:stronglysplit1optiv}
$S_{\tau\nabla g,\tau\nabla f}$ and 
$S_{\tau\nabla f,\tau\nabla g}$ are $r_S(\tau)$-Lipschitz 
continuous, where
\begin{equation}
\label{e:constDR}
r_S(\tau)=\min\left\{\frac{1+r_R(\tau)}{2},\dfrac{\beta+\tau^2\rho}
{\beta+\tau\beta\rho+\tau^2\rho}\right\}
\in\left]0,1\right[
\end{equation}
and $r_R$ is defined in \eqref{e:constPR}.
In particular, the optimal step-size and the minimum in 
\eqref{e:constDR} are
\begin{equation}
	\label{eq:taudr}
(\tau^*,r_{S}(\tau^*))=
\begin{cases}
\left(\sqrt{\frac{\alpha}{\rho}},\frac{1}{1+\sqrt{\alpha\rho}}\right),\quad&\text{if}\:\:\beta\leq
{4\alpha};\\
\left(\sqrt{\frac{\beta}{\rho}},\frac{2}{2+\sqrt{\beta\rho}}\right),\quad&\text{otherwise}.
\end{cases}
	\end{equation}
\end{enumerate}
\end{proposition}

\begin{table}
\hspace{-0.2cm}\begin{tabular}{|c|c|c|c|}
\hline
Algorithm & Linear convergence rate & 
Optimal step-size & 
Optimal linear \\
 &  as a 
function of $\tau$&  &  convergence rate  \\
\hline
\hline
 EA with & 
 $\max\big\{|1-\tau\rho|,|1-\tau(\beta^{-1}+\alpha^{-1})|\big\}$ & 
 $\dfrac{2}{\rho+\alpha^{-1}+\beta^{-1}}$ & 
 $\dfrac{\alpha^{-1}+\beta^{-1}-\rho}
{\alpha^{-1}+\beta^{-1}+\rho}$\\
 $\tau\in\left]0,2/(\alpha^{-1}+\beta^{-1})\right[$& & 
&\\
\hline
FBS with $\prox_g$ and $\nabla f$ & 
$\max\big\{|1-\tau\rho|,|1-\tau\alpha^{-1}|\big\}$ & 
$\dfrac{2}{\rho+\alpha^{-1}}$ & $\dfrac{\alpha^{-1}-\rho}
{\alpha^{-1}+\rho}$\\
 and $\tau\in\left]0,2\alpha\right[$ & & &\\
\hline
FBS with $\prox_f$ and $\nabla g$ & $\dfrac{1}{1+\tau\rho}$ & 
$2 \beta$ & $\dfrac{1}{1+2 \beta\rho}$ \\
  and $\tau\in\left]0,2\beta\right]$ & & &\\
\hline
PRS with $\tau>0$ & $\max\left\{\dfrac{1-\tau\rho}{1+\tau\rho},
\dfrac{\tau\alpha^{-1}-1}{\tau\alpha^{-1}+1}\right\}$ & 
$\dfrac{1}{1+2 \beta\rho}$ & 
${\dfrac{1-\sqrt{\alpha\rho}}{1+\sqrt{\alpha\rho}}}$\\
 & & &\\
\hline
DRS with $\tau>0^{*}$ & 
$\min\left\{\frac{1+r_R(\tau)}{2},\dfrac{\beta+\tau^2\rho}
{\beta+\tau\beta\rho+\tau^2\rho}\right\}$ &
$\begin{cases}
\sqrt{\frac{\alpha}{\rho}},\quad&\text{if}\:\:\beta\leq
{4\alpha};\\
\sqrt{\frac{\beta}{\rho}},\quad&\text{otherwise}.
\end{cases}$&
$\begin{cases}
\frac{1}{1+\sqrt{\alpha\rho}},\quad&\text{if}\:\:\beta\leq
{4\alpha};\\
\frac{2}{2+\sqrt{\beta\rho}},\quad&\text{otherwise}.
\end{cases}$
\\
\hline
\end{tabular}
\caption{Convergence rates and optimal step-sizes 
for several {first-order} schemes when considering the 
minimization 
problem $\min_{x\in \HH} f(x) + g(x)$, where $f$ (resp. $g$) is 
convex differentiable with $\alpha^{-1}$ (resp. 
$\beta^{-1}$)-Lipschitz gradient and $f$ is $\rho$-strongly 
convex. $^*$ The results on DRS are new.  \label{tab:sumrates}}
\end{table}

\noindent 
The Lipschitz 
{constants}  
of the operators $G_{\nabla g+\nabla f}$ and $T_{\nabla g,\nabla 
f}$ are {a} consequence of \cite[Theorem~3.1]{Taylor18} (see 
also 
\cite[Fact~3]{Yin20} for a geometric interpretation). We provide an 
alternative shorter and more direct proof of 
Proposition~\ref{p:stronglysplit1opt}\eqref{p:stronglysplit1opti}-\eqref{p:stronglysplit1optii}
 in Appendix~\ref{app:stronglysplit1opt}, in which we use some 
techniques from \cite[Section~2.1.3]{Nesterov}. The Lipschitz 
constant of $T_{\nabla f,\nabla g}$ is a direct consequence of 
Proposition~\ref{p:stronglysplit1}\eqref{p:stronglysplit1iiv} and
\eqref{e:constPR} 
is obtained in \cite[Theorem~2]{BoydGiss}, 
which improves several constants in the literature. The Lipschitz 
constant in \eqref{e:constDR} is obtained by combining 
\cite[Theorem~2]{BoydGiss} and \cite[Theorem~5.6]{Giselsson2017}.
\begin{remark}
\begin{enumerate}
\item {When $\rho\approx 0$, 
\eqref{eq:taufb2}
	justifies} the classical choice $\tau^*\approx 2\alpha$. This 
case arises naturally in several inverse problems and, in particular,
in sparse image restoration which is studied in detail in 
Section~\ref{sec:IIIC}.

\item Note that the Lipschitz continuous constants obtained in 
Proposition~\ref{p:stronglysplit1opt}\eqref{p:stronglysplit1opti}
and \ref{p:stronglysplit1opt}\eqref{p:stronglysplit1optii} are strictly 
lower than the constants 
obtained in 
Proposition~\ref{p:stronglysplit1}\eqref{p:stronglysplit1i}
and \ref{p:stronglysplit1}\eqref{p:stronglysplit1ii} in the cocoercive 
case, as it can be 
verified in Figure~\ref{fig:ratecocoopt}.
\item 
From Figure~\ref{fig:ratecocoopt}, we observe the benefit of 
the refinement of convergence rates in the optimization framework 
(dashed line)
with respect to the cocoercive case (solid line) in
all methods at exception of $T_{\tau\nabla f,\tau\nabla g}$, whose 
rate is the same. We also 
observe that in general Peaceman-Rachford iterations $R_{\tau\nabla 
g,\tau\nabla f}$ has the better convergence rate for several 
configurations of $(\alpha,\beta,\rho)$.\\
\end{enumerate}
\end{remark}

In the case when  $g=0\in\mathscr{C}_0^{1,1}(\HH)$,
Problem~\ref{prob:mainsplitopt} reduces to minimize $f$ over $\HH$
and $G_{\nabla g+\nabla f}=
T_{\nabla g,\nabla f}=G_{\nabla f}$
and $T_{\nabla f,\nabla g}=S_{\nabla f,\nabla g}=S_{\nabla g,\nabla f}
=\prox_{f}$.
Therefore, by taking $\beta\to+\infty$ in 
Proposition~\ref{p:stronglysplit1opt}, we recover 
the following known results
(see also \cite[Proposition~5.2]{Giselsson2017} and 
\cite[Proposition~4.39]{bauschke2011:Convex_analysis}).

\begin{proposition}
	\label{p:B0optst}
Let $\tau\in\RPP$, $\alpha\in\RPP$, 
$\rho\in\left]0,\alpha^{-1}\right]$, and suppose that 
$f\in\mathscr{C}^{1,1}_{1/\alpha}(\HH)$ and that $f$ is $\rho$-strongly 
convex. Then, the following hold.
\begin{enumerate}
\item \label{p:B0optsti} Suppose that $\tau\in\left]0,2\alpha\right[$. 
Then $G_{\tau\nabla f}$ is 
$r_{G_0}(\tau)$-Lipschitz 
continuous, where 
\begin{equation}
r_{G_0}(\tau):=\max\big\{|1-\tau\rho|,
|1-\tau\alpha^{-1}|\big\}\in\left]0,1\right[.
\end{equation}
\item
\label{p:B0optstii} $\prox_{\tau f}$ is ${r}_J(\tau)$-Lipschitz 
continuous, 
where 
\begin{equation}
{r}_{J}(\tau):=\dfrac{1}{1+\tau\rho}\in\left]0,1\right[.
\end{equation}
\end{enumerate}
\end{proposition}

\subsection{Comparison of algorithms}
Since 
Problem~\ref{prob:mainsplitopt} is equivalent to minimize $\alpha 
f+\alpha g$ over $\HH$, we can assume $\alpha=1$. Set 
$\Omega=\RPP\times\left]0,1\right[$ and denote 
\begin{equation}
\label{e:defoptrate}
\begin{cases}
r_G^*(\beta,\rho)=\frac{1+\beta^{-1}-\rho}{1+\beta^{-1}-\rho}\\
r_{T_1}^*(\rho)=\frac{1-\rho}{1+\rho}\\
r_{T_2}^*(\beta,\rho)=\frac{1}{1+2\beta \rho}\\
r_{R}^*(\rho)=\frac{1-\sqrt{\rho}}{1+\sqrt{\rho}}\\
r_{S}^*(\beta,\rho)=
\begin{cases}
\frac{1}{1+\sqrt{\rho}},&\text{if}\:\:\beta\le 4;\\
\frac{2}{2+\sqrt{\beta\rho}},&\text{if}\:\:\beta> 4.
\end{cases}
\end{cases}
\end{equation}
Observe that $r_G^*(\beta,\rho)=r_G(\tau^*)$, 
$r_{T_1}^*(\rho)=r_{T_1}(\tau^*)$, 
$r_{T_2}^*(\beta,\rho)=r_{T_2}(\tau^*)$, $r_{R}^*(\rho)=r_{R}(\tau^*)$,
and $r_{S}^*(\beta,\rho)=r_{S}(\tau^*)$ are the optimal rates 
obtained in Proposition~\ref{p:stronglysplit1opt} when $\alpha=1$.
We have the following comparisons.
\begin{lemma}
\label{l:1}
Let $(\beta,\rho)\in\Omega$. Then
$r_G^*(\beta,\rho)>r_{T_1}^*(\rho)>r_{R}^*(\rho)$.
\end{lemma}
\begin{proof}
Set 
\begin{equation}
\phi\colon (t,\rho)\mapsto \frac{t-\rho}{t+\rho}=1-\frac{2}{1+t/\rho}
\end{equation}
 and note that, for every $(t,\rho)\in\RPP\times\left]0,1\right[$, 
 $\phi(\cdot,\rho)$ is strictly increasing on $\RPP$ and 
$\phi(t,\cdot)$ is strictly decreasing on $\left]0,1\right[$. Noting that 
$\sqrt{\rho}>\rho$, the 
result follows 
from $r_G^*(\beta,\rho)=\phi(1+\beta^{-1},\rho)> 
\phi(1,\rho)=r_{T_1}^*(\rho)$
and $r_{T_1}^*(\rho)=\phi(1,\rho)>\phi(1,\sqrt{\rho})=r_{R}^*(\rho)$.
\end{proof}
We conclude from Lemma~\ref{l:1} that PRS is always more efficient 
than the algorithms governed by 
operators $T_{\tau^*\nabla g,\tau^*\nabla f}$ and $G_{\tau^*(\nabla 
g+\nabla f)}$ for solving Problem~\ref{prob:mainsplitopt}. Therefore,
it is enough to compare $r_{R}^*(\rho)$, $r_{T_2}^*(\beta,\rho)$,
and $r_{S}^*(\beta,\rho)$.
\begin{lemma}
\label{l:comp}
Let $(\beta,\rho)\in\Omega$. The following hold:
\begin{enumerate}
\item 
\label{l:compi}
$r_{T_2}^*(\beta,\rho)<r_{R}^*(\rho)\:\Leftrightarrow\:
\beta>4\:\:\text{and}\:\: 
\beta\rho\in\left]\eta(\beta),
\eta(\beta)^{-1}\right[$,
where
\begin{equation}
\label{e:defeta}
\eta(\beta)=\frac{1-\sqrt{1-4\beta^{-1}}}{1+\sqrt{1-4\beta^{-1}}}\in\left]0,1\right[.
\end{equation}
\item 
\label{l:compii}
$r_{S}^*(\beta,\rho)<r_{R}^*(\rho)\:\:\Leftrightarrow\:\:\beta>16\:\:\:
\text{and}\:\:\:\rho
<1-8(\frac{\sqrt{\beta}-2}{\beta}).$
\item 
\label{l:compiii}
Suppose that $\beta>4$. Then
$r_{S}^*(\beta,\rho)<r_{T_2}^*(\beta,\rho)\:\:\Leftrightarrow
\:\:\rho<\frac{1}{16\beta}.$
\end{enumerate}
\end{lemma}
\begin{proof}
\ref{l:compi}:
Note that
\begin{align}
r_{T_2}^*(\beta,\rho)-r_{R}^*(\rho)=\frac{1}{1+2\beta \rho}-
\frac{1-\sqrt{\rho}}{1+\sqrt{\rho}}
=\frac{1+\sqrt{\rho}-(1-\sqrt{\rho})(1+2\beta \rho)}{(1+2\beta 
\rho)(1+\sqrt{\rho})}
=\frac{2\sqrt{\rho}\beta(\rho-\sqrt{\rho}+\beta^{-1})}{(1+2\beta 
\rho)(1+\sqrt{\rho})}.
\end{align}
Hence,  $r_{T_2}^*(\beta,\rho)<r_{R}^*(\rho)$ is equivalent to 
$1-4\beta^{-1}> 0$ and $\sqrt{\rho}\in\left]\eta_1,\eta_2\right[$, where 
$\eta_1=(1-\sqrt{1-4\beta^{-1}})/2$ and 
$\eta_2=(1+\sqrt{1-4\beta^{-1}})/2$, which yields the result after 
simple computations.

\ref{l:compii}:
It is clear from \eqref{e:defoptrate} that, when $\beta\le 4$, 
$r_S^*(\beta,\rho)\ge r_R^*(\rho)$. Hence, by assuming that 
$\beta>4$, we have
\begin{align}
r_S^*(\beta,\rho)-r_R^*(\rho)&=\frac{2}{2+\sqrt{\beta \rho}}-
\frac{1-\sqrt{\rho}}{1+\sqrt{\rho}}
=\frac{4\sqrt{\rho}-\sqrt{\beta \rho}(1-\sqrt{\rho})}{(2+\sqrt{\beta 
\rho})(1+\sqrt{\rho})}
=\frac{\sqrt{\rho}(4-\sqrt{\beta}+\sqrt{\beta \rho})}{(2+\sqrt{\beta 
\rho})(1+\sqrt{\rho})}.
\end{align}
We observe that, for every $\beta\le 16$, we have
$r_S^*(\beta,\rho)\ge r_R^*(\rho)$ and, if $\beta>16$,
$r_S^*(\beta,\rho)< r_R^*(\rho)$ if and only if 
$\sqrt{\rho}<1-4/\sqrt{\beta}$, from which the result follows.

\ref{l:compiii}: Since
\begin{align}
r_S^*(\beta,\rho)-r_{T_2}^*(\beta,\rho)=
\frac{2}{2+\sqrt{\beta\rho}}-\frac{1}{1+2\beta \rho}
=
\frac{\sqrt{\beta\rho}(4\sqrt{\beta\rho}-1)}{(2+\sqrt{\beta\rho})(1+2\beta 
\rho)},
\end{align}
the proof is complete.
\end{proof}
Now, by using Lemma~\ref{l:comp}, we can conclude which algorithm 
has the lower convergence 
rate depending on the parameters $(\beta,\rho)\in\Omega$.
In Figure~\ref{fig:compthrate} we illustrate the efficiency regions 
thus derived { and Table~\ref{tab:bestalgo} summarizes the result 
of Lemma~\ref{l:comp}.}
\begin{proposition}
\label{prop:comp}
Let $(\beta,\rho)\in\Omega$ and let $\eta$ be the 
function defined in \eqref{e:defeta}. Then, the following hold:
\begin{enumerate}
\item Suppose that $\beta>4$ and that $\rho\in I(\beta)$, where
$$I(\beta)=
\left[\frac{\max\{1/16,\eta(\beta)\}}{\beta},\frac{1}{\beta\eta(\beta)}\right].$$
 Then 
$r_{T_2}^*(\beta,\rho)\le
\min\{r_S^*(\beta,\rho),r_R^*(\rho)\}$.
\item Suppose that $\beta>16$ and that $\rho<\chi(\beta)$, where
$$\chi(\beta)=\min\left\{\frac{1}{16\beta},1-8\frac{\sqrt{\beta}-2}{\beta}\right\}.$$
Then $r_{S}^*(\beta,\rho)<
\min\{r_{T_2}^*(\beta,\rho),r_R^*(\rho)\}$.
\end{enumerate}
In any other case, we have $r_R^*(\rho)\le
\min\{r_{T_2}^*(\beta,\rho),r_{S}^*(\beta,\rho)\}$.
\end{proposition}

\section{Numerical experiments}
\label{sec:exp}
The theoretical results provided in the previous sections are now 
illustrated on standard data processing examples with different 
{levels} of complexity: Piecewise-constant 
denoising and image restoration. {The Matlab codes associated 
with the following experiments are available on Nelly Pustelnik 
website 
{(\href{https://perso.ens-lyon.fr/nelly.pustelnik/Software/toolbox_ratecompare_v-1.0.zip}{link})}.}

\subsection{Piecewise constant denoising}
\label{sec:V-B}

Piecewise constant denoising (also referred as change-point 
detection) is a very well documented problem of signal processing 
literature and it is of interest for numerous signal processing 
{applications} going from  
genomics~\cite{Vert_J_2010_nips_fast_dmc} 
to geophysics studies~\cite{Pascal_B_2020_j-at_para_ffn}.


\begin{table}[t]
\centering
\begin{tabular}{|c|c|c|}
\hline
Region of parameters $(\beta, \rho)$ & Algorithm with the best 
rate\\
\hline
\hline
$\Omega_1=\Big\{(\beta,\rho) \,\vert\, \beta>4\;\;\mbox{and}\;\; 
\rho \in \big 
[\frac{\max\{1/16,\eta(\beta)\}}{\beta},\frac{1}{\beta\eta(\beta)}\Big]\Big\}$&
 FBS with $\prox_{\tau f}$\\
\hline
$\Omega_2=\Big\{(\beta,\rho) \,\vert\, \beta>4\;\;\mbox{and}\;\; 
\rho \in \Big[\frac{1}{16\beta} ,1-8\frac{\sqrt{\beta}-2}{\beta}\Big 
]\Big\}$ & DRS\\
\hline
$\Omega\setminus(\Omega_1\cup \Omega_2 )$ & PRS\\
\hline
\end{tabular}
\caption{Best rate algorithms among EA, FBS with prox or 
gradient activation of the strongly convex function, DRS, PRS 
when $\alpha=1$ when considering the minimization 
problem $\min_{x\in \HH} f(x) + g(x)$, where $f$ (resp. $g$) is 
convex differentiable with $\alpha^{-1}$ (resp. 
$\beta^{-1}$)-Lipschitz gradient and $f$ is $\rho$-strongly 
convex.\label{tab:bestalgo}}
\end{table}

\begin{figure*}[t]
\centering
\begin{tabular}{cc}
\hspace{-0.7cm}\includegraphics[height = 
6cm]{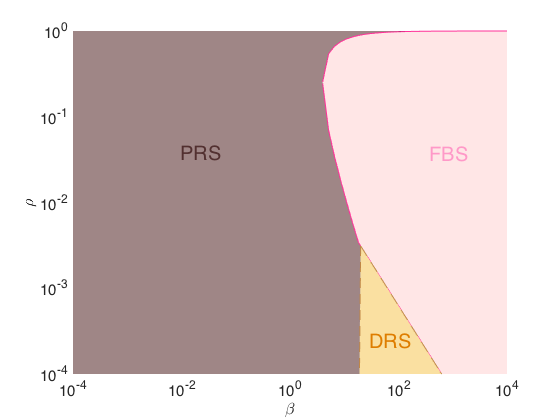} & 
\raisebox{0.1cm}{\includegraphics[width = 7.7cm, height = 
5.6cm]{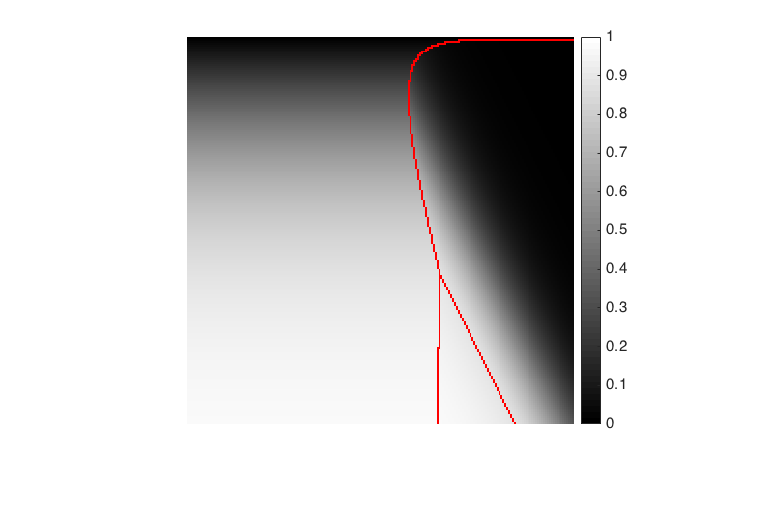}} \\
\end{tabular}
\caption{{(Left) Regimes where PRS or FBS or DRS achieves a 
better rate according to Proposition~\ref{prop:comp} when 
$\alpha=1$ as a function of $(\beta,\rho)$. (Right) Optimal 
numerical rates and associated regions.}\label{fig:compthrate}}
\end{figure*}	
The standard formulation is dedicated to piecewise constant signal 
$\overline{x}\in \mathbb{R}^N$ degraded with a 
Gaussian noise $\varepsilon\sim \mathcal{N}(0,\sigma^2\mathbb{I})$, whose degraded version is denoted $z = \overline{x} + \varepsilon$. An 
illustration of $\overline{x}$ (resp. $z$) is provided in solid black line (resp. gray) in Figure~\ref{fig:piecewiseex1}~(top). 

The estimation of a piecewise 
constant signal $\widehat{x}$ from degraded data $z$  
has  been  addressed  by  several strategies going from Cusum  
procedures~\cite{Basseville:M:1993},  hierarchical  Bayesian  inference  
frameworks~\cite{Lavielle_M_2001_sp_app_mcmc},  or  functional optimization formulations involving 
$\ell_1$-norm or the $\ell_0$-pseudo-norm of the first 
differences of the signal (see 
e.g.~\cite{Frecon_J_2017_j-ieee-tsp_bay_sl2} and references 
therein). In  the latter context, we consider the minimization 
problem:
{\begin{equation}
\label{eq:l2hubL}
\minimize{x\in\RR^N} {\frac{1}{2} \Vert x - z \Vert_2^2+\chi 
h_{\mu}(Dx)},
\end{equation}}
where {$D\in \RR^{N-1\times N}$} denotes the {first-order} 
discrete difference operator
{$$
(\forall n\in\{1,\ldots,N-1\})\quad (Dx)_n = \frac{1}{2}(x_n - x_{n-1})
$$}
and  $h_{\mu}\colon\RR^{N-1}\to\RR$ denotes the Huber loss of 
parameter $\mu>0$, {which is a smooth 
approximation of the 
$\ell_1$-norm defined by (see, e.g., 
\cite[Example~2.5]{combettesGlaudin})
\begin{equation}
\label{e:defHuber}
h_{\mu}\colon (\zeta_i)_{1\le i\le 
m}\mapsto \sum_{i=1}^{N-1} \phi_{\mu} 
(\zeta_i)\quad\text{and}\quad
\phi_{\mu}\colon \zeta
\mapsto\left\{\begin{array}{ll}{|\zeta|-\frac{\mu}{2},} & {\text { if 
}|\zeta|>\mu}; \\ 
	{\frac{|\zeta|^{2}}{2\mu},} & {\text { if }|\zeta| \leqslant 
	\mu}.\end{array}\right.
\end{equation}
Note that, since
$$
\phi'_{\mu}\colon \zeta
\mapsto\left\{\begin{array}{ll}{\frac{\zeta}{|\zeta|},} & {\text { if 
		}|\zeta|>\mu}; \\ 
	{\frac{\zeta}{\mu},} & {\text { if }|\zeta| \leqslant 
		\mu},\end{array}\right.
$$
we have  $h_{\mu}\in \mathscr{C}^{1,1}_{1/\mu}(\RR^{N-1})$.}   By 
setting
$f = \frac{1}{2} \Vert \cdot - z \Vert_2^2$ and 
{$g = \chi h_{\mu}\circ D$}, \eqref{eq:l2hubL} is a particular 
instance of Problem~\ref{prob:mainsplitopt}, where 
$f$ is $\rho = 1$ strongly convex, $\alpha=1$, and 
{$\beta = \frac{\mu}{{\chi}\Vert D\Vert^2}$} and it can be 
solved by the following two schemes:
\begin{enumerate}
\item[1-] \textbf{EA: } Use  
$G_{\tau(\nabla 
g+\nabla f)}$ with the
step-size $\tau^*$ in \eqref{eq:taugrad}.
\item[2-] \textbf{FBS: } Use 
$T_{\tau\nabla f,\tau\nabla 
	g}$  with the step-size 
$\tau^*$ in \eqref{eq:taufb3}.
\end{enumerate}
Moreover, the proximity operator of $h_{\mu}$ can be computed 
explicitly via
\begin{equation}
(\forall \tau>0)\quad \prox_{\tau h_{\mu}}\colon (\zeta_i)_{1\le i\le 
	m}\mapsto (\prox_{\tau\phi_{\mu}}\zeta_i)_{1\le i\le m},
	\end{equation}
where
	\begin{equation} 
\prox_{\tau \phi_{\mu}}\colon \zeta
\mapsto\left\{\begin{array}{ll}{\zeta -  \frac{\tau\zeta}{|\zeta|},} & 
	{\text { if }|\zeta|>\tau + \mu};\\ 
	{\frac{\mu\zeta}{\tau+\mu},} & {\text { if }|\zeta| \leqslant \tau + 
		\mu}.\end{array}\right.
\end{equation}
However, the proximity 
operator of {$h_{\mu}\circ D$} is not explicit because of the 
influence of 
operator {$D$}.
By exploiting the separable structure of $h_{\mu}$, we obtain the 
following equivalent formulation of 
\eqref{eq:l2hubL}:
{\begin{equation}
	\label{eq:l2hubLsplit}
	\min_{x\in\HH} \frac{1}{2} \Vert x - z \Vert_2^2+\chi 
	h_{\mathbb{I}_1}(D_{\mathbb{I}_1}x)+\chi 
	h_{\mathbb{I}_2}(D_{\mathbb{I}_2}x),
\end{equation}}
where $\mathbb{I}_1 = \{1,3,\ldots\}$ and 
$\mathbb{I}_2 = \{2,4,\ldots\}$ are the sets of odd and even 
indices and, for $k\in\{1,2\}$,
$h_{\mathbb{I}_k}(y_{\mathbb{I}_k}) = \sum_{i\in \mathbb{I}_k} 
\phi_{\mu}(y_i)$, and {$D_{\mathbb{I}_k}\in 
\mathbb{R}^{|\mathbb{I}_k| 
\times N}$} denotes the sub-matrix of {$D$} associated with the 
$\mathbb{I}_k$ rows. Since
{$D_{\mathbb{I}_1}D_{\mathbb{I}_1}^{\top}=\textrm{Id}/2$ and 
$D_{\mathbb{I}_2}D_{\mathbb{I}_2}^{\top}=\textrm{Id}/2$}, 
the split formulation \eqref{eq:l2hubLsplit} allows for the following
closed form expressions of the proximity operator of 
{$h_{\mathbb{I}_k} 
\circ D_{\mathbb{I}_k}$} (see 
\cite[Proposition~23.25]{bauschke2011:Convex_analysis})
{$$
(\forall k\in\{1,2\} )(\forall \tau>0)\quad 
\prox_{\tau h_{\mathbb{I}_k} \circ D_{\mathbb{I}_k}}\colon 
z\mapsto  z 
- 2L_{\mathbb{I}_k}^\top 
\big(\Id- \prox_{\frac{\tau}{2}
h_{\mathbb{I}_k}}\big)(D_{\mathbb{I}_k}z),
$$}
where $\prox_{\frac{\tau}{2} 
h_{\mathbb{I}_k}}\colon (\zeta_i)_{i\in \mathbb{I}_k}
\mapsto (\prox_{\frac{\tau}{2}\phi_{\mu}}\zeta_i)_{i\in \mathbb{I}_k}$.
By setting {$\widetilde{f} = \frac{1}{2} \Vert \cdot - z \Vert_2^2 + 
\chi h_{\mathbb{I}_2}(D_{\mathbb{I}_2}\cdot)$} and 
{$\widetilde{g} 
= \chi h_{\mathbb{I}_1}(D_{\mathbb{I}_1}\cdot)$,} 
we write \eqref{eq:l2hubLsplit} as Problem~\ref{prob:mainsplitopt},
where $\widetilde{f}$ is $\rho=1$ strongly convex, 
{$\alpha=\frac{\mu}{\mu+ 
{\chi}\Vert 	D_{\mathbb{I}_2}\Vert^2}$, and $\beta = 
\frac{\mu}{{\chi}\Vert D_{\mathbb{I}_1}\Vert^2}$.} 
This approach gives raise to 4 alternative methods 
for solving \eqref{eq:l2hubLsplit}.

\begin{enumerate}
\item[3-] \textbf{FBS 2: }  Use  
$T_{\tau\nabla 
\widetilde{g},\tau\nabla \widetilde{f}}$ with the
step-size $\tau^*$ in \eqref{eq:taufb2}.
\item[4-] \textbf{FBS 3: }  Use   
$T_{\tau\nabla \widetilde{f},\tau\nabla 
\widetilde{g}}$  with the
step-size $\tau^*$ in \eqref{eq:taufb3}.
\item[5-]  \textbf{PRS:}  Use  
$R_{\tau\nabla 
\widetilde{f},\tau\nabla 
\widetilde{g}}$  with the step-size $\tau^*$ in \eqref{eq:taupropti}.
\item[6-] \textbf{DRS: }  Use  
$S_{\tau\nabla \widetilde{f},\tau\nabla \widetilde{g}}$  with the 
step-size $\tau^*$ in 
\eqref{eq:taudr}.
\end{enumerate}

We consider an approximation of the unique solution $\widehat{x}$ to
\eqref{eq:l2hubL}, by applying PRS with a large number of iterations. 
In view of Section~\ref{ssec:probalgo}, 1-EA, 2-FBS, 3-FBS2, and 4-FBS3 are 
initialized with  $x_0 = z$, while using
$$z = \prox_{\gamma f}(x_n) 
\Leftrightarrow (\Id+ \gamma \nabla f) 
y_n = x_n$$ 
{proximal-based} procedures 5-PRS
and 6-DRS are initialized by $x_0 = z +\tau \nabla f(z)$, 
in order to provide similar initializations.

The numerical and theoretical convergence rate are displayed in 
Figure~\ref{fig:piecewiseex1} for different 
settings of $\mu$ and $\chi$ 
leading to sharper or smoother estimates depending of the 
configuration. When $\mu = 10^{-4}$ the performance are 
{similar to 
what is} expected for $\ell_1$-minimization. 

From Figure~\ref{fig:piecewiseex1}~(bottom), {we can observe 
that 
PRS iterations provide the best theoretical and experimental rates 
when the optimal step-size is selected. DRS iterations also provide a 
good behavior}, while EA and FBS strategies relying on the splitting 
$f = \frac{1}{2} \Vert \cdot - z \Vert_2^2$ and 
{$g = \chi h_{\mu}\circ D$} appears less efficient than the one 
involving 
the 
splitting {$\widetilde{f} = \frac{1}{2} \Vert \cdot - z \Vert_2^2 + 
\chi h_{\mathbb{I}_2}(D_{\mathbb{I}_2}\cdot)$ and $\widetilde{g} 
= \chi h_{\mathbb{I}_1}\circ D_{\mathbb{I}_1}$.} Similar conclusion 
can be observed from Figure~\ref{fig:piecewiseex1}~(top), where the 
optimal solution is reached after 100 iterations for DRS (light brown) 
and PRS  (dark brown) while gradient based procedures require much 
more iterations. This is especially true when $\mu$ is small, leading to 
a large Lipschitz constant and, thus, to a small step-size for 
gradient-based algorithms. Moreover, observe that our results are 
consistent with Proposition~\ref{prop:comp} illustrated in 
Figure~\ref{fig:compthrate}. Since $\alpha=\rho=1$, we verify that 
PRS is the most efficient algorithm and that EA and FBS are not 
competitive. 
\begin{figure*}[h!]
\label{fig:l2hubL}
\hspace{-1cm}\begin{tabular}{c c}
\includegraphics[width 
=7.8cm]{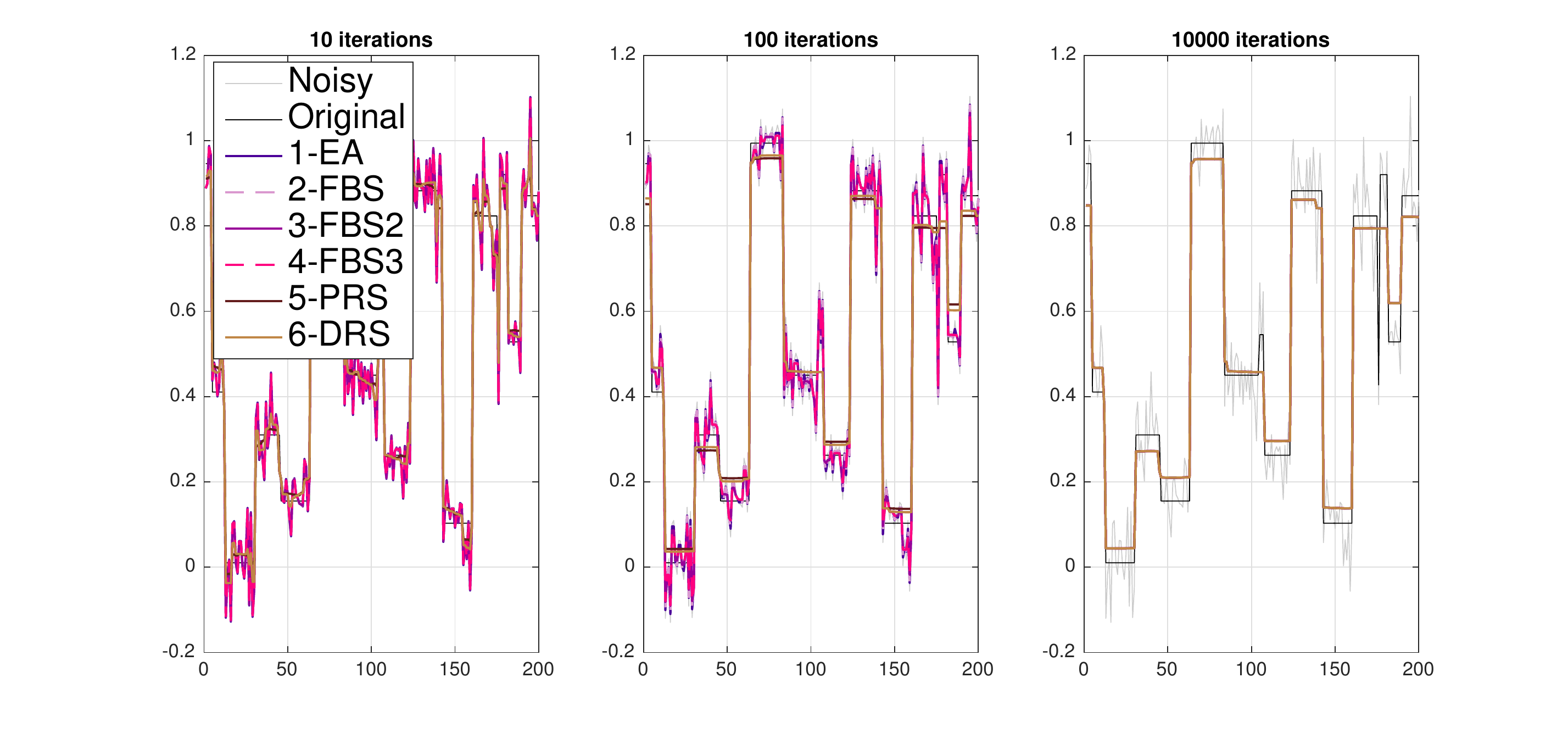} & 
\includegraphics[width =7.8cm]{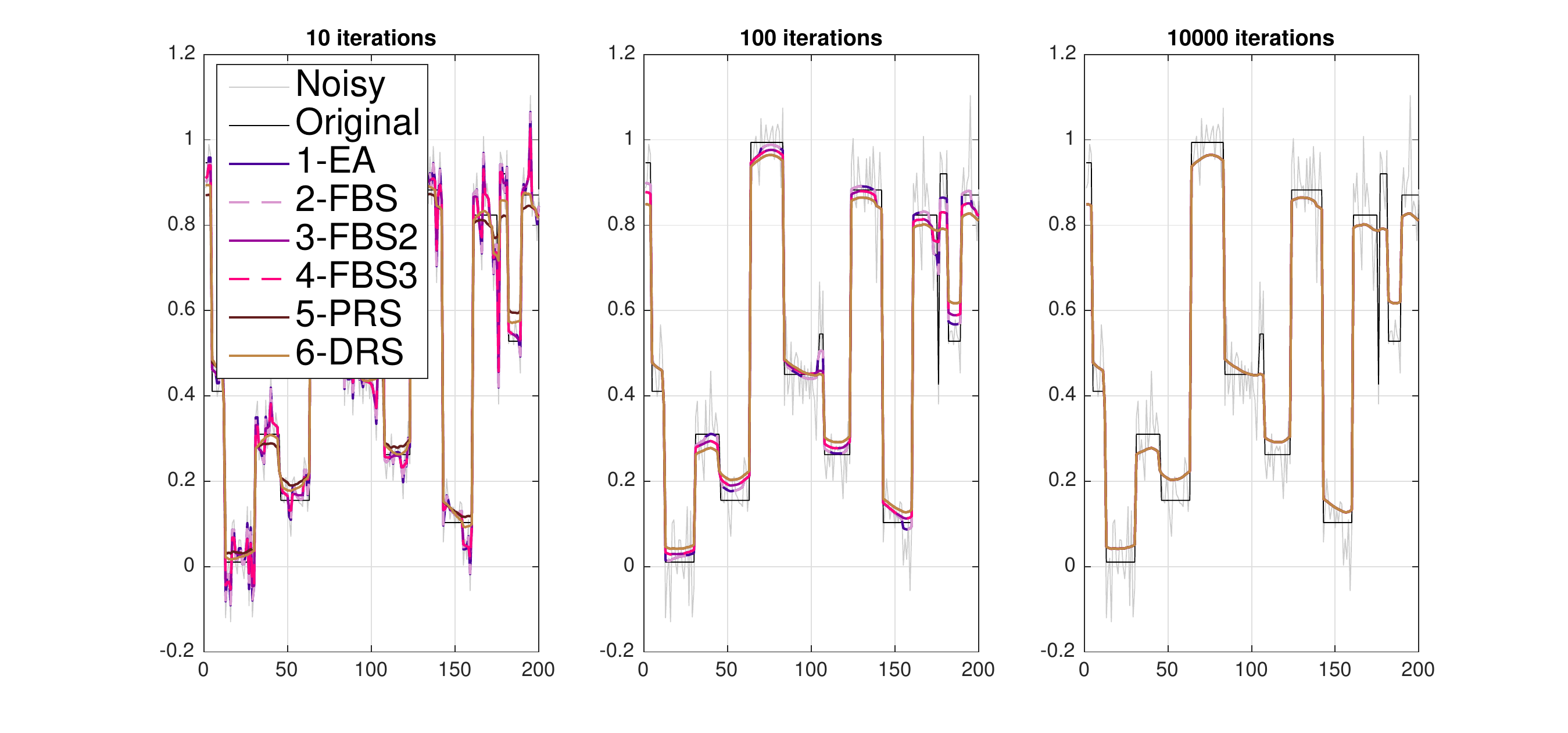} 
\\{(a) Original/degraded/reconstructed signals} & {(b) 
Original/degraded/reconstructed signals}\\
\includegraphics[width = 
7.8cm]{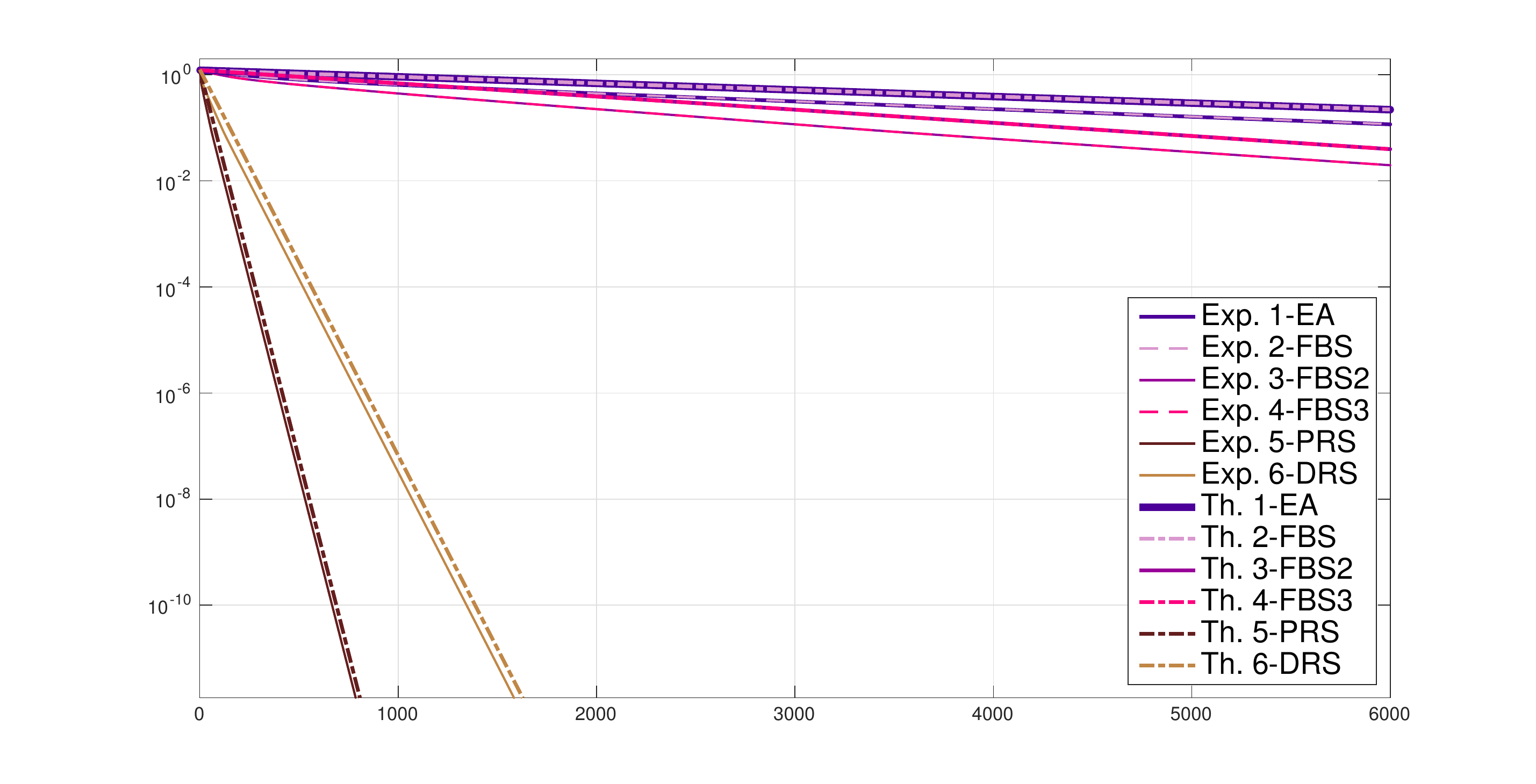} & 
\includegraphics[width = 
7.8cm]{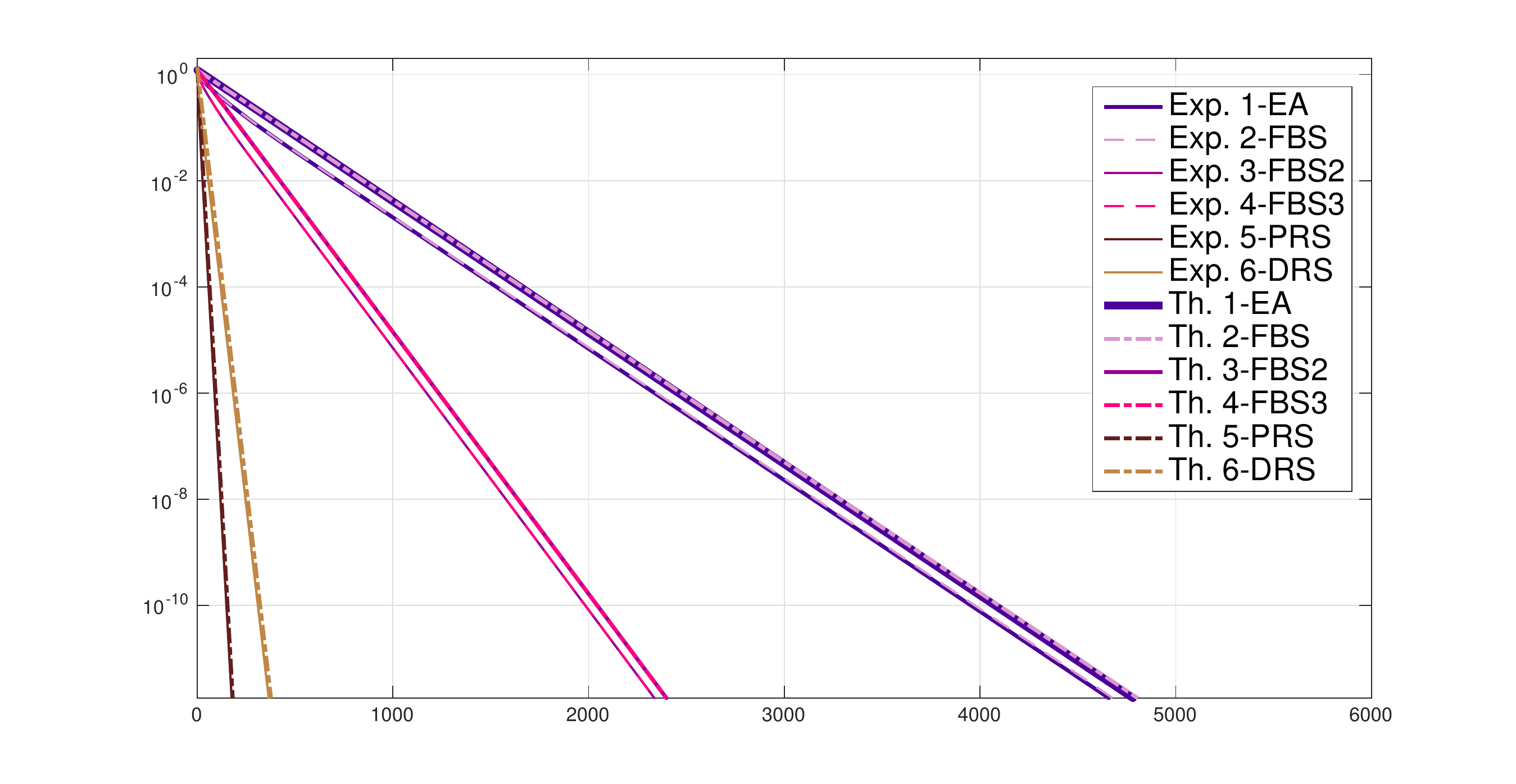}\\
{(c) Errors vs Iterations} & {(d) Errors vs Iterations}
\end{tabular}
\caption{{\small Piecewise constant denoising estimates 
after 10, 100, and 10000 iterations with 
$\chi=0.7$ and $\mu = 0.0001$ (a) and $\chi=0.7$ and $\mu = 
0.002$ (b). We can observe that the piecewise 
constant estimate is obtained after 100 iterations for 
DRS or PRS while EA or FBS requires much more iterations.
We also exhibit the experimental and theoretical {errors} 
associated with 
each implemented 
{method} for optimal step-size $\tau$ with respect to iteration 
number (c-d). The behavior is in accordance 
{with} the results observed on the first row. 
\label{fig:piecewiseex1}}}
\end{figure*}

%

\subsection{Image restoration}
\label{sec:IIIC}
Another classical signal processing problem is image restoration that 
consists in recovering an image $\overline{x}\in \mathbb{R}^N$ with 
$N$ pixels from degraded observations $z = A\overline{x} + 
\varepsilon$, where $A\in \RR^{M\times 
N}$ {with $M\ge N$} and $\varepsilon\sim 
\mathcal{N}(0,\sigma^2\mathbb{I})$ is a 
white Gaussian noise. 
A standard penalization imposes the sparsity of the coefficients 
resulting from a linear transform such as a wavelet transform 
\cite{Pustelnik_N_20016_j-w-enc-eee_wav_bid} and the restoration 
can then be achieved by solving
\begin{equation}
\label{eq:l2AhubL}
\underset{{x\in\HH}}{\textrm{minimize}}\; \frac{1}{2} \Vert Ax - z 
\Vert_2^2+\chi h_{\mu}(Wx),
\end{equation} 
where $\chi>0$ is the regularization parameter, $W$ denotes a 
weighted wavelet transform, and $h_{\mu}$ is the 
Huber penalization of parameter $\mu>0$ {defined in 
\eqref{e:defHuber}.}

Following Proposition~\ref{prop:comp}, we propose to evaluate the 
theoretical and the experimental rates for the following algorithmic 
schemes, where $f = \frac{1}{2} \Vert A\cdot - z \Vert_2^2$ 
and $g = \chi h_{\mu} \circ W$:
\begin{enumerate}
\item[1-] \textbf{FBS: }Use 
$T_{\tau\nabla f,\tau\nabla 
g}$  with the
step-size $\tau^*$ in \eqref{eq:taufb3}.
\item[2-]  \textbf{PRS:} Use 
$R_{\tau\nabla 
f,\tau\nabla 
g}$  with the step-size $\tau^*$ in \eqref{eq:taupropti}.
\item[3-] \textbf{DRS: }   Use 
$S_{\tau\nabla f,\tau\nabla g}$  with the step-size $\tau^*$ in 
\eqref{eq:taudr}.
\end{enumerate}
In this context, by denoting $\lambda_{\max}$ and $\lambda_{\min}$
the largest and lowest eigenvalues of $A^{\top}A$, 
$\rho = \lambda_{\min}$ is the strong 
convexity parameter of $f$, 
$\lambda_{\max}=\alpha^{-1}$ is the Lipschitz constant of $\nabla f$, 
and  $\beta^{-1} =\frac{\chi}{\mu}$ is the Lipschitz constant of $\nabla 
g$. 
The results are displayed in 
Figure~\ref{fig:restoresults} for an image with $N=2^{12}$ pixels when 
$A$ is a random Gaussian matrix {of size $4900\times 4096$} 
with $\lambda_{\min}=0.0022$ 
and 
$\lambda_{\max}=1$. The results are obtained considering three
values of $\chi\in\{0.001,0.04,10\}$ and $\mu= 1$ in order to consider 
three instances (green dots) in
the three different efficiency regions (displayed in brown, orange, and 
pink).

\begin{figure*}
\centering
\vspace{-1cm}
\begin{tabular}{cc}
Original & \\
\vspace{-3cm}\includegraphics[height =  
2.4cm]{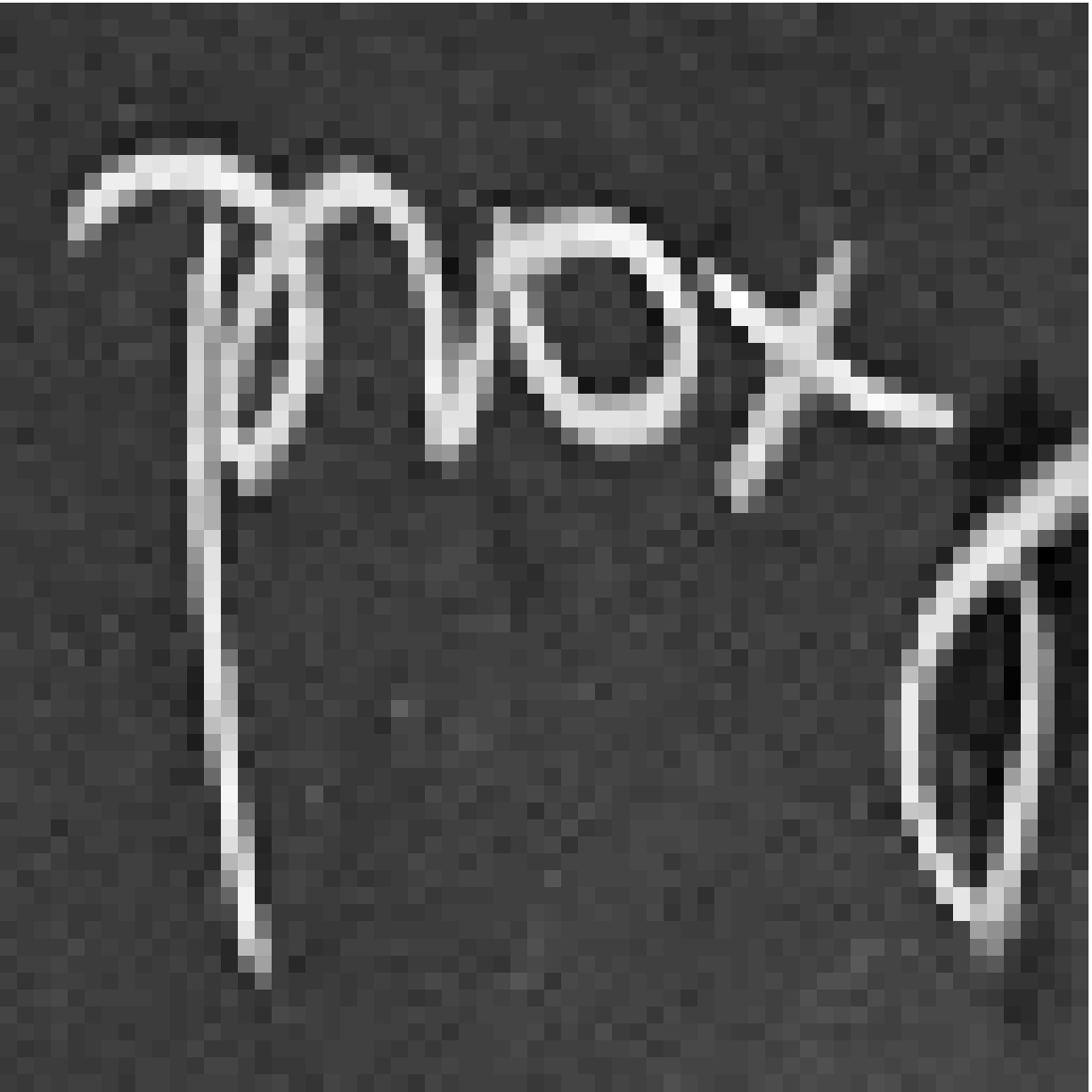}&\raisebox{-3cm}{\includegraphics[height = 
5.8cm]{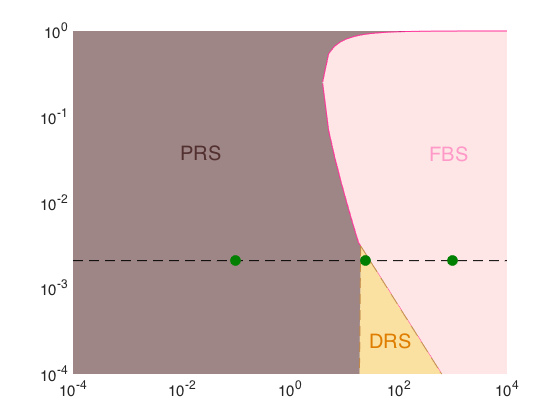}}\\
$A^{-1}z$ & \\
\includegraphics[height =  2.4cm]{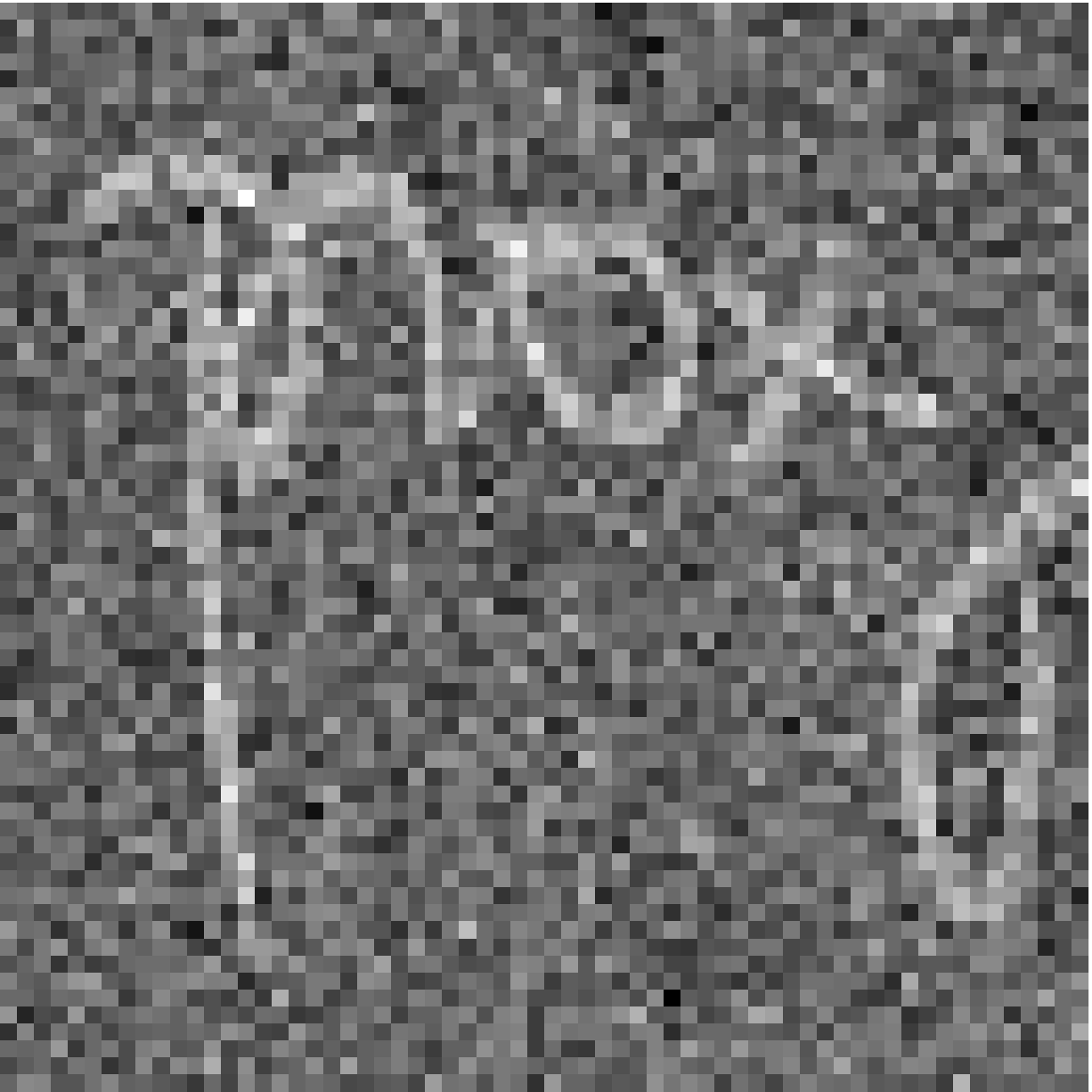}&\\
&\\
\multicolumn{2}{c}{Image restored with PRS with $\beta =0.1$ and 
convergence behavior for the different schemes.}\\
\raisebox{1.3cm}{\includegraphics[height =  
2.4cm]{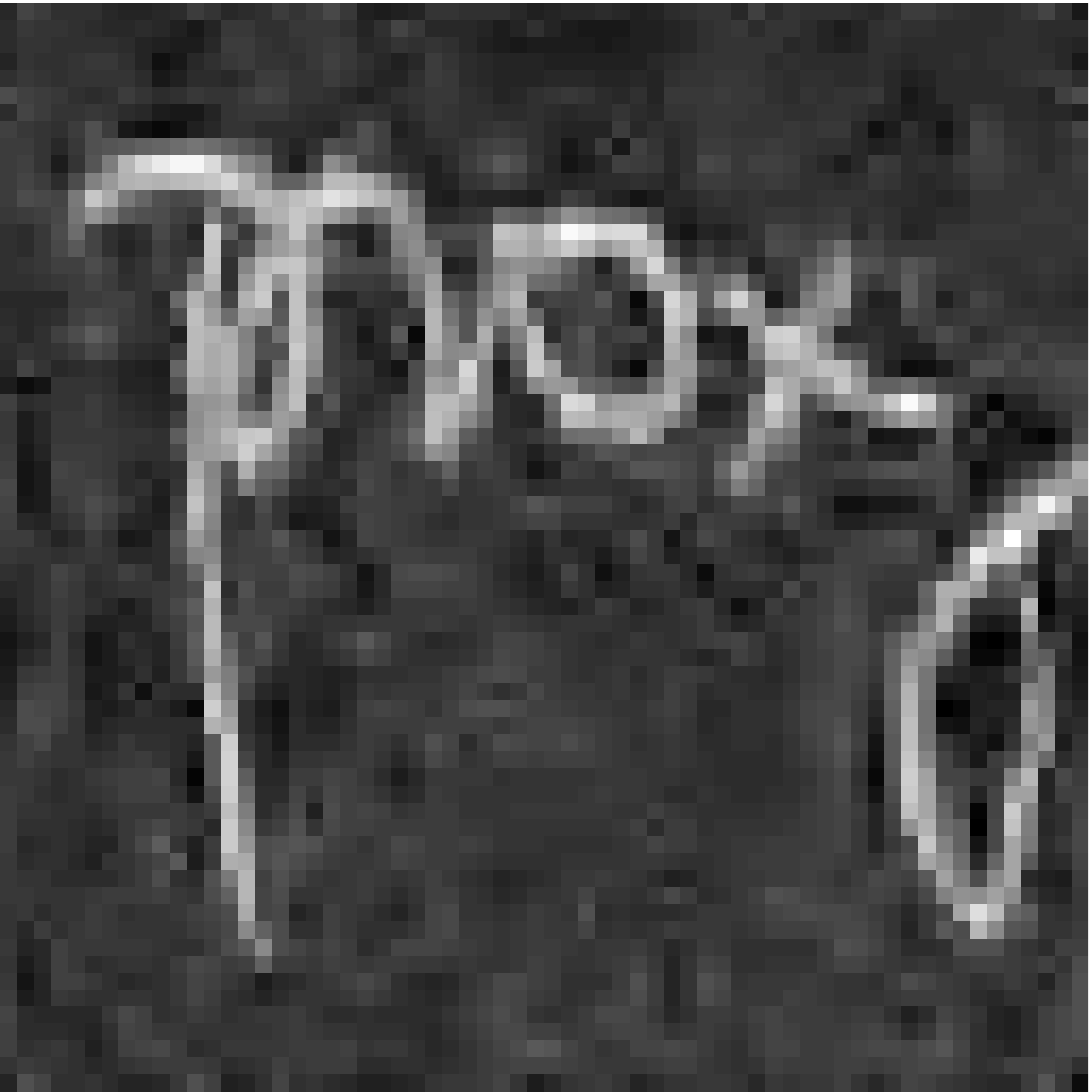}}&\includegraphics[height =  
3.8cm]{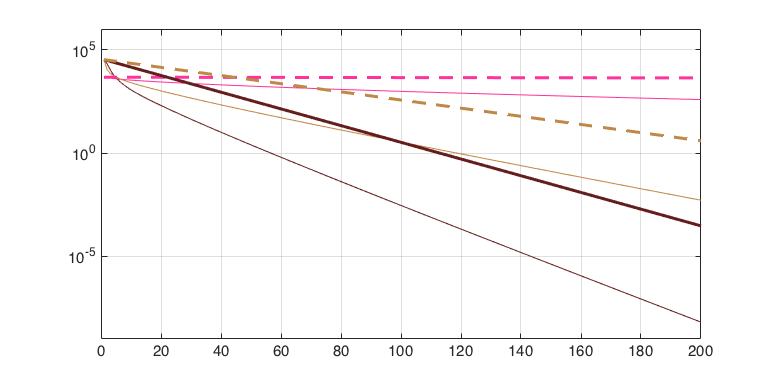}\\
\multicolumn{2}{c}{Image restored with DRS with $\beta =25$  and 
convergence behavior for the different schemes.}\\
\raisebox{1.3cm}{\includegraphics[height =  
2.4cm]{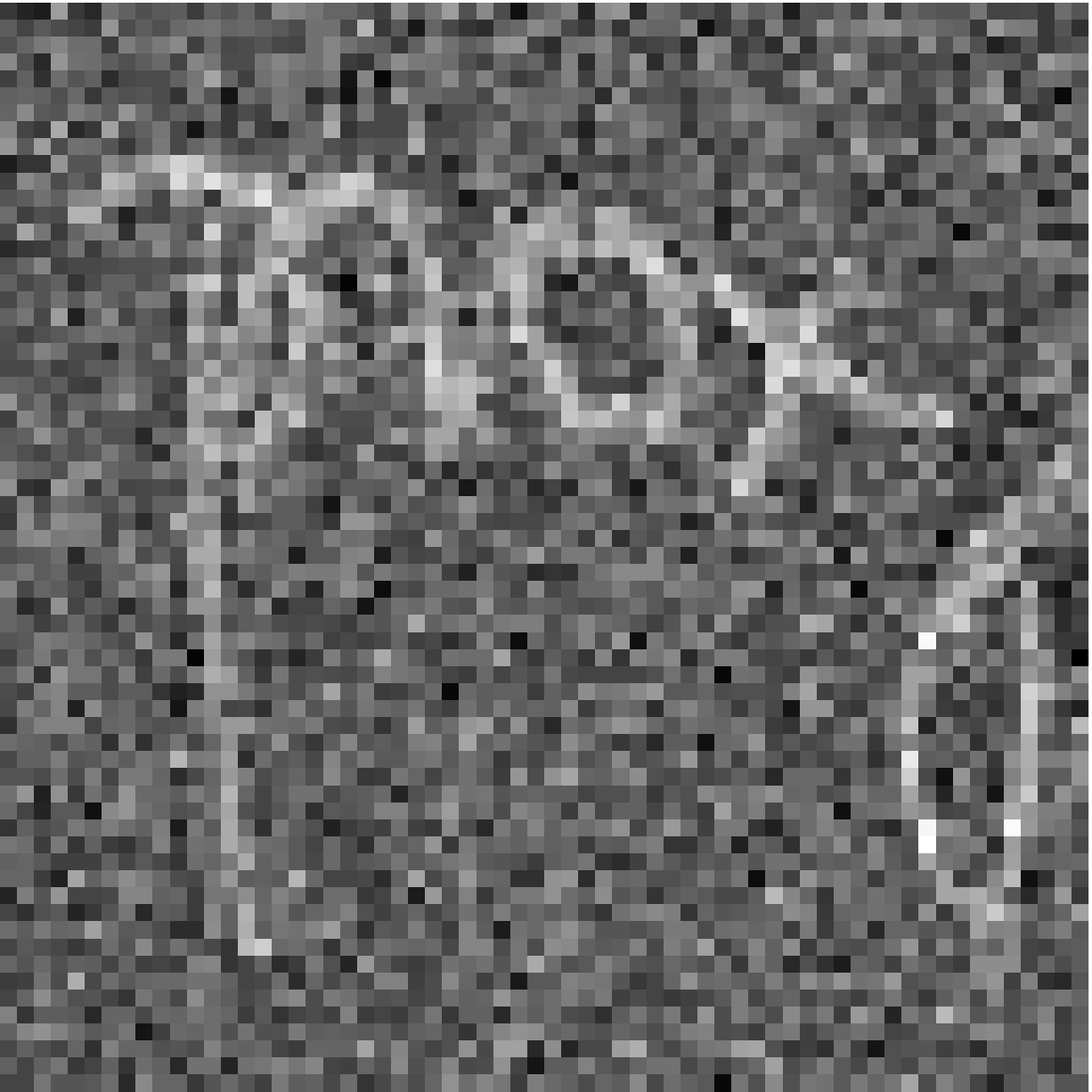}}&\includegraphics[height =  
3.8cm]{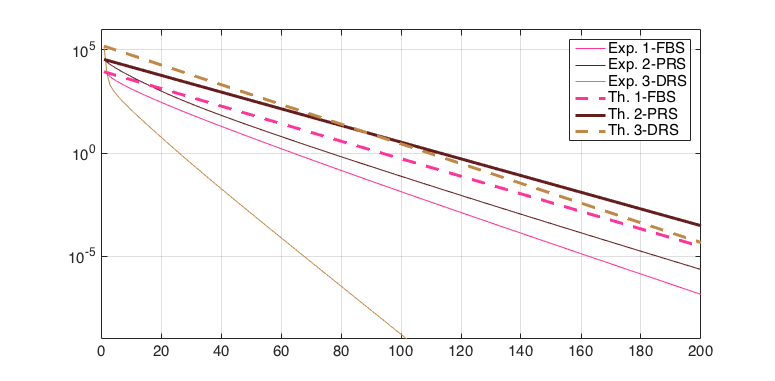}\\
\multicolumn{2}{c}{Image restored with FBS with $\beta =1000$  and 
convergence behavior for the different schemes.}\\
\raisebox{1.3cm}{\includegraphics[height =  2.4cm]{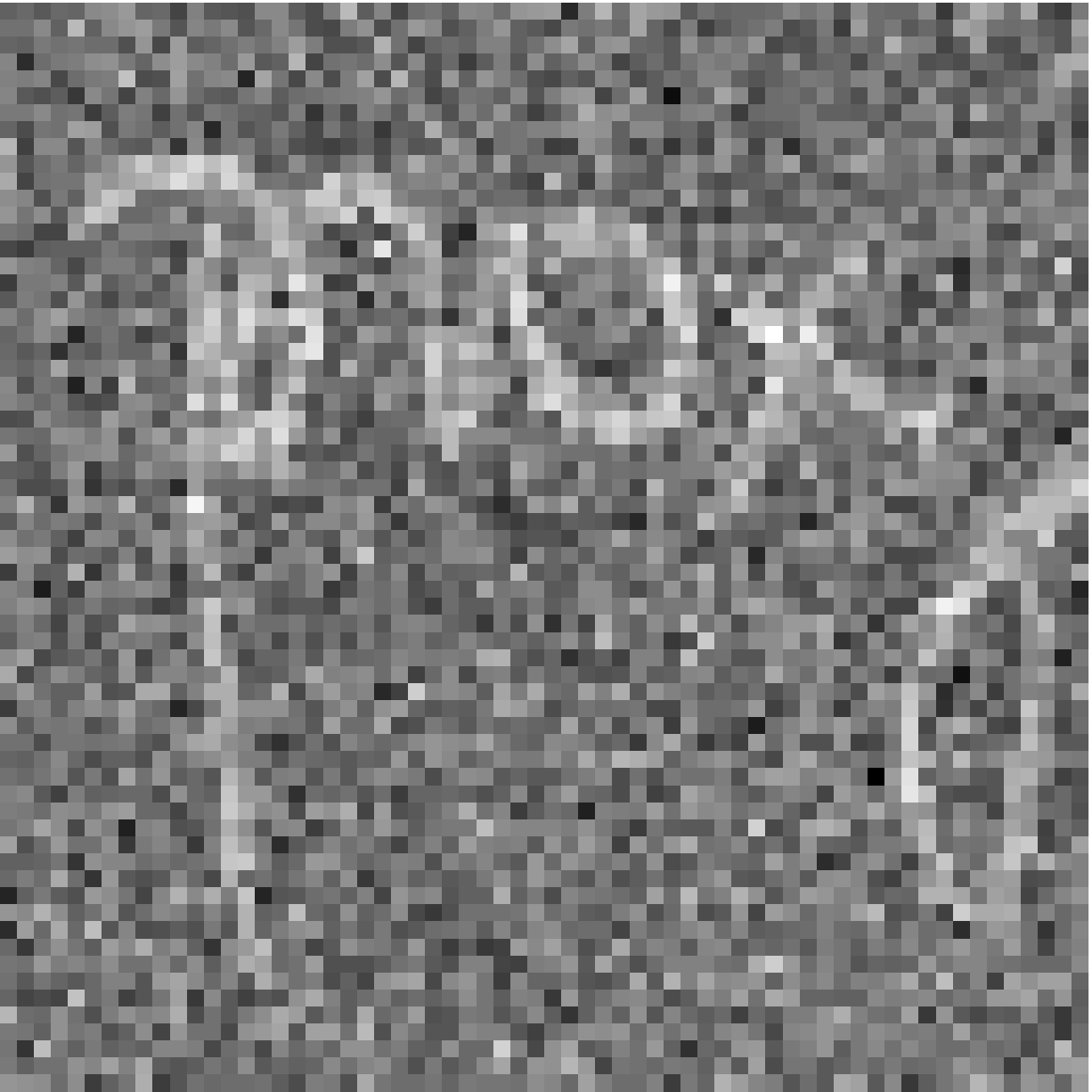}}&
\includegraphics[height =  3.8cm]{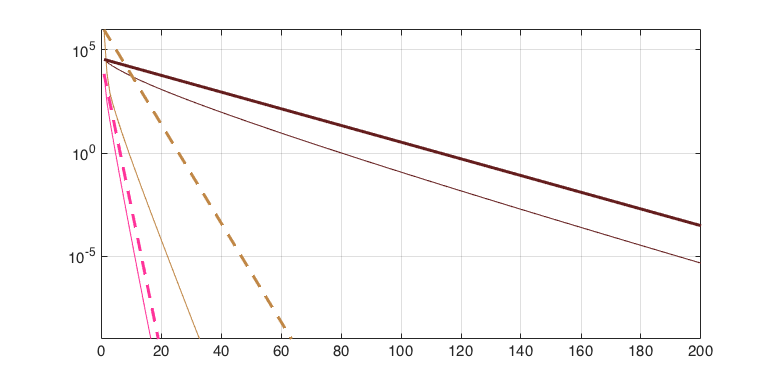}\\
\end{tabular}
\caption{Image restoration example considering a random matrix. 
The figure at the top right includes in 
Figure~\ref{fig:compthrate} a dashed black line and three green dots
representing the cases explored in this experiment ($\rho = 
\lambda_{\min}=0.0022$, $\alpha=1/\lambda_{\max}=1$, $\mu=1$,
and $\beta=\mu/\chi\in\{0.1,25,1000\}$). We verify that the theoretical 
and numerical errors decay as predicted in 
Proposition~\ref{prop:comp}.  \label{fig:restoresults}}
\end{figure*}

%
The first set of experiments (second row in 
Figure~\ref{fig:restoresults}) displays the 
results obtained with $\chi=10$ ($\beta = 0.1$). In this case, 
PRS achieves a better rate according to the theoretical 
study in Proposition~\ref{prop:comp} (brown region in 
Figure~\ref{fig:compthrate}) and our numerical experiments 
confirm this result. The third (fourth) row displays the results 
obtained with $\chi=0.04$ (resp. $\chi=0.001$) and, thus, 
$\beta = 25$ (resp. $\beta = 1000$), associated with the orange (resp. 
pink) 
efficiency 
region where DRS (resp. FBS) leads to {the} 
best theoretical rate. {In this context where the strong convexity 
constant is small, the fit between theoretical and numerical 
convergence behaviour is not as tight.} The best 
restoration results are obtained when $\chi=10$, in which case PRS 
performs better.
\small
\bibliographystyle{acm}
\bibliography{references}
%
{\section{Cocoercivity and strong monotonicity in 
Example~\ref{ex:coco1}}
\label{sec:cocomon}
Consider the operators $\mathcal{A}$ and $\mathcal{B}$
defined in \eqref{e:maxmonops},
and fix
$(x_1,u_1)$ and $(x_2,u_2)$ in $\RR^n\times\RR^p$. 
For every $(y_1,v_1)\in\mathcal{A}(x_1,u_1)$ and 
$(y_2,v_2)\in\mathcal{A}(x_2,u_2)$, since 
$h^*$ is $1/L$-strongly 
convex \cite[Theorem~18.15]{bauschke2011:Convex_analysis},
we have
\begin{align}
\scal{(y_1,v_1)-(y_2,v_2)}{(x_1,u_1)-(x_2,u_2)}&=
\scal{\nabla f(x_1)-\nabla 
f(x_2)}{x_1-x_2}-\eta\|x_1-x_2\|^2\nonumber\\
&\qquad +\scal{v_1+\eta u_1-(v_2+\eta 
u_2)}{u_1-u_2}-\eta\|u_1-u_2\|^2\nonumber\\
&\ge (\mu-\eta)\|x_1-x_2\|^2+(1/L-\eta)\|u_1-u_2\|^2\nonumber\\
&\ge(\min\{\mu,1/L\}-\eta)\|(x_1,u_1)-(x_2-u_2)\|^2,
\end{align}
which implies the $(\min\{\mu,1/L\}-\eta)$-strong convexity of 
$\mathcal{A}$.
Moreover, in the case when $h$ is also $\rho$-strongly convex, 
we have $h^*\in\mathscr{C}_{1/\rho}^{1,1}(\RR^p)$, 
$\mathcal{A}\colon (x,u)\mapsto(\nabla 
f({x})-\eta x,\nabla h^*({u})-\eta u)$ is single valued and
\begin{align}
\label{e:aux}
\scal{\mathcal{A}(x_1,u_1)-\mathcal{A}(x_2,u_2)}
{(x_1,u_1)-(x_2,u_2)}&=
\scal{\nabla f(x_1)-\nabla 
f(x_2)}{x_1-x_2}-\eta\|x_1-x_2\|^2\nonumber\\
&\qquad +\scal{\nabla h^*(u_1)-\nabla 
h^*(u_2)}{u_1-u_2}-\eta\|u_1-u_2\|^2\nonumber\\
&\ge\frac{1}{\|A\|^2}\|\nabla f(x_1)-\nabla 
f(x_2)\|^2+\rho\|\nabla h^*(u_1)-\nabla 
h^*(u_2)\|^2.
\end{align}
On the other hand, the strong monotonicity of $\nabla f$
and $\nabla h^*$ imply $\|\nabla f(x_1)-\nabla f(x_2)\|\ge 
\mu\|x_1-x_2\|$ and $\|\nabla h^*(u_1)-\nabla h^*(u_2)\|\ge 
(1/L)\|u_1-u_2\|$, which yield
\begin{align}
\|\mathcal{A}(x_1,u_1)-\mathcal{A}(x_2,u_2)\|&\le 
\|\nabla f(x_1)-\nabla f(x_2)\|+\eta\|x_1-x_2\|+\|\nabla 
h^*(u_1)-\nabla 
h^*(u_2)\|+\eta\|u_1-u_2\|\nonumber\\
&\le (1+\eta/\mu)\|\nabla f(x_1)-\nabla f(x_2)\|+(1+\eta L)\|\nabla 
h^*(u_1)-\nabla h^*(u_2)\|.
\end{align}
Therefore, it follows from \eqref{e:aux} that
$\mathcal{A}$ is 
$(\min\{\rho,\|A\|^{-2}\}/(1+\eta\max\{(1/\mu,L\})^2)$-cocoercive.
Finally,
\begin{align}
\scal{\mathcal{B}(x_1,u_1)-\mathcal{B}(x_2,u_2)}
{(x_1,u_1)-(x_2,u_2)}&=
\eta\|x_1-x_2\|^2+
\scal{D^{\top}(x_1-x_2)}{x_1-x_2}\nonumber\\
&\qquad+\eta\|u_1-u_2\|^2-\scal{D(x_1-x_2)}{u_1-u_2}\nonumber\\
&= \eta\|(x_1,u_1)-(x_2-u_2)\|^2\nonumber\\
&\ge 
\frac{\eta}{\|\mathcal{B}\|^2}\|\mathcal{B}(x_1,u_1)-
\mathcal{B}(x_2-u_2)\|^2,
\end{align}
which implies the $\eta/\|\mathcal{B}\|^2$-cocoercivity of 
$\mathcal{B}$.}

\section{Averaged nonexpansive constants in the case 
$\rho=0$}
\begin{proposition}
	\label{p:2}
Let $\tau>0$. In the context of Problem~\ref{prob:mainsplit}, the 
following hold:
\begin{enumerate}
\item 
\label{p:2i}
Suppose that 
$\tau\in\left]0,2\beta\alpha/(\beta+\alpha)\right[$. Then
$G_{\tau(\mathcal{A}+\mathcal{B})}$ is $\mu_G(\tau)$-averaged 
nonexpansive, 
where
\begin{equation}
\label{e:constantG2}
\mu_G(\tau):=\frac{\tau(\beta+\alpha)}
{2\beta\alpha}\in\left]0,1\right[.
\end{equation}
\item 
\label{p:2ii}
Suppose that $\tau\in\left]0,2\alpha\right[$. Then
$T_{\tau\mathcal{B},\tau\mathcal{A}}$ is 
$\mu_T(\tau)$-averaged nonexpansive, 
where
\begin{equation}
\label{e:muT}
\mu_T(\tau):=\frac{2\tau(\beta+\alpha)}{4\beta\alpha+\tau(4\alpha-\tau)}
\in\left]0,1\right[.
\end{equation}
\item
\label{p:2iii}
$R_{\tau\mathcal{B},\tau\mathcal{A}}$  is 
$\mu_R(\tau)$-averaged nonexpansive, 
where
\begin{equation}
\mu_R(\tau):=\dfrac{\tau}{\frac{\alpha\beta}
	{\alpha+\beta}+\tau}\in\left]0,1\right[.
\end{equation}
\item
\label{p:2iv}
$S_{\tau\mathcal{B},\tau\mathcal{A}}$   is 
$\mu_S(\tau)=\frac{\mu_R(\tau)}{2}$-averaged nonexpansive.
\end{enumerate}	
\begin{proof}
\ref{p:2i}: It follows from 
\cite[Proposition~4.12]{bauschke2011:Convex_analysis} that
$\mathcal{A}+\mathcal{B}$ is 
$(\beta^{-1}+\alpha^{-1})^{-1}=\alpha\beta/(\alpha+\beta)$-cocoercive.
 
The result thus follows from
\cite[Proposition~4.39]{bauschke2011:Convex_analysis}.

\ref{p:2ii}: Since $\tau\mathcal{B}$ is $\beta/\tau$-cocoercive, it 
follows from \cite[Proposition~5.2]{Giselsson2017} and
\cite[Proposition~4.39]{bauschke2011:Convex_analysis} that 
$J_{\tau\mathcal{B}}$ and $G_{\tau\mathcal{A}}$ are 
$\alpha_{\mathcal{B}}=\tau/(2(\tau+\beta))\in\left]0,1/2\right[$ and 
$\alpha_{\mathcal{A}}=\tau/(2\alpha)$-averaged nonexpansive, 
respectively.
Hence, we deduce from \cite[Proposition~2.4]{Yamada} that 
$T_{\tau\mathcal{B},\tau\mathcal{A}}=J_{\tau\mathcal{B}}\circ 
G_{\tau\mathcal{A}}$ is averaged with constant 
$(\alpha_{\mathcal{B}}+
	\alpha_{\mathcal{A}}-2\alpha_{\mathcal{B}}
	\alpha_{\mathcal{A}})/(1-\alpha_{\mathcal{B}}
	\alpha_{\mathcal{A}})$ which leads the result after simple 
	computations.

\ref{p:2iii}: Since $\tau\mathcal{A}$ and $\tau\mathcal{B}$ are 
$\alpha/\tau$- and $\beta/\tau$-cocoercive, respectively, it follows 
from \cite[Proposition~5.3]{Giselsson2017} that  $R_{\tau 
\mathcal{A}}=2J_{\tau 
\mathcal{A}}-\Id$ and $R_{\tau \mathcal{B}}=2J_{\tau 
\mathcal{B}}-\Id$ are $\tau/(\tau+\alpha)$ and $\tau/(\tau+\beta)$ 
averaged nonexpansive, respectively. Hence, since
 $R_{\tau \mathcal{B},\tau \mathcal{A}}=R_{\tau \mathcal{B}}\circ 
 R_{\tau \mathcal{A}}$, the averaging constant is obtained from
\cite[Proposition~2.4]{Yamada} as in \ref{p:2ii}.

\ref{p:2iv}: Since $S_{\tau,\mathcal{B},\mathcal{A}}=(\Id+R_{\tau 
\mathcal{B}}\circ 
R_{\tau \mathcal{A}})/2$ we deduce the result from 
\ref{p:2iii} and 
\cite[Proposition~4.40]{bauschke2011:Convex_analysis}.
\end{proof}

\end{proposition}
\section{Proof of Proposition~\ref{p:stronglysplit1}}
\label{app:stronglysplit1}

\ref{p:stronglysplit1i}: Set 
$\mathcal{M}=\mathcal{A}+\mathcal{B}$,
fix 
$\tau\in\left]0,2\beta\alpha/(\beta+\alpha)\right[\subset\left]0,
2\min\{\beta,\alpha\}\right[$, fix
$x$ and $y$ in $\HH$. From the $\rho$-strong monotonicity and 
$\alpha$-cocoercivity of $\mathcal{A}$, we have, for every 
$\lambda\in\left]0,1\right[$,
\small
\begin{align}
\label{eq:proofp7a}
\scal{\mathcal{M}x-\!\mathcal{M}y}{x-y}\!=\!
\scal{\mathcal{B}x-\mathcal{B}y\!}{\!x-y}\!+\!
\scal{\mathcal{A}x-\!\mathcal{A}y}{\!x-y}
\ge \beta\|\mathcal{B}x-\mathcal{B}y\|^2+\lambda\alpha
\|\mathcal{A}x-\mathcal{A}y\|^2+(1-\lambda)\rho\|x-y\|^2\nonumber.
\end{align}
Hence, noting that, for every $\varepsilon>0$,  
$\|\mathcal{M}x-\mathcal{M}y\|^2\le(1+\varepsilon)
\|\mathcal{B}x-\mathcal{B}y\|^2+
(1+\varepsilon^{-1})\|\mathcal{A}x-\mathcal{A}y\|^2,$
we deduce
\begin{align*}
\|G_{\tau\mathcal{M}}x-G_{\tau\mathcal{M}}y\|^2
&=\|x-y\|^2-2\tau
\scal{\mathcal{M}x-\mathcal{M}y}{x-y}+
\tau^2\|\mathcal{M}x-\mathcal{M}y\|^2\nonumber\\
&\le\|x-y\|^2-2\tau\beta\|\mathcal{B}x-\mathcal{B}y\|^2-
2\tau\lambda\alpha\|\mathcal{A}x-\mathcal{A}y\|^2\nonumber\\
&\qquad  -
2\tau\rho(1-\lambda)\|x-y\|^2+
\tau^2\|\mathcal{M}x-\mathcal{M}y\|^2\nonumber\\
&\le (1-2\tau\rho 
(1-\lambda))\|x-y\|^2 -\tau(2\beta-\tau(1+\varepsilon))
\|\mathcal{B}x-\mathcal{B}y\|^2\nonumber\\&\qquad-
\tau(2\lambda\alpha-\tau(1+\varepsilon^{-1}))
\|\mathcal{A}x-\mathcal{A}y\|^2.
\end{align*}
Thus, the result follows by setting $\varepsilon=(2\beta-\tau)/\tau>0$
and $\lambda=\frac{\tau\beta}{\alpha(2\beta-\tau)}\in\left]0,1\right[$.

\ref{p:stronglysplit1ii}: Fix $\tau\in\left]0,2\alpha\right[$. It follows 
from \ref{p:stronglysplit1i} in the limit case when 
$\mathcal{B}=0$ ($\beta\to+\infty$) that 
$G_{\tau\mathcal{A}}$ is 
$\omega_{T_1}(\tau)$-Lipschitz continuous (see also 
\cite[Fact~7]{Yin20}). 
Hence, the result follows from
$T_{\tau\mathcal{B},\tau\mathcal{A}}=J_{\tau\mathcal{B}}\circ 
G_{\tau\mathcal{A}}$ and the nonexpansivity of $J_{\tau\mathcal{B}}$.

\ref{p:stronglysplit1iiv}: Fix $\tau\in\left]0,2\beta\right]$. 
It follows 
from \cite[Proposition~23.13]{bauschke2011:Convex_analysis} that 
$J_{\tau\mathcal{A}}$ is 
$\omega_{T_2}(\tau)$-Lipschitz continuous. The result follows from
$T_{\tau\mathcal{A},\tau\mathcal{B}}=J_{\tau\mathcal{A}}\circ 
G_{\tau\mathcal{B}}$ and the nonexpansivity of 
$G_{\tau\mathcal{B}}$.

\ref{p:stronglysplit1iii}: First note that 
\cite[Theorem~7.2]{Giselsson2017} implies that $R_{\tau 
\mathcal{A}}=2J_{\tau\mathcal{A}}-\Id$ is $\omega_R(\tau)$-Lipschitz 
continuous.
Now, since $R_{\tau\mathcal{B}}$ is nonexpansive, we obtain that
$R_{\tau\mathcal{B}}
R_{\tau\mathcal{A}}$ and $R_{\tau\mathcal{A}}
R_{\tau\mathcal{B}}$ are also $\omega_R(\tau)$-Lipschitz continuous. 

\ref{p:stronglysplit1iv}:
Since
 $S_{\tau\mathcal{B},\tau\mathcal{A}}=
(\Id+R_{\tau\mathcal{B},\tau\mathcal{A}})/2$
  and  
$S_{\tau\mathcal{A},\tau\mathcal{B}}=(\Id+R_{\tau\mathcal{A},\tau\mathcal{B}})/2$,
 this result is a consequence of 
\cite[Lemma~3.3\,\&\,Theorem~5.6]{Giselsson2017}, and 
\ref{p:stronglysplit1iii}. 

In all the cases, the minima are obtained via 
simple computations.

\section{Proof of Proposition~\ref{p:stronglysplit1opt}}
\label{app:stronglysplit1opt}
\ref{p:stronglysplit1opti}: Set $h=f+g$.
Since $g$ is convex and Fr\'echet differentiable and 
$f\in\mathscr{C}_{1/\beta}^{1,1}(\HH)$ is $\rho$-strongly convex, we 
obtain that
$\phi=h-\rho\|\cdot\|^2/2$ is convex and Fréchet differentiable. 
Moreover, since $\nabla f$ and 
$\nabla g$ are $\alpha^{-1}$-Lipschitz continuous and 
$\beta^{-1}$-Lipschitz continuous, we have that $\nabla h=\nabla 
f+\nabla g$ is $\gamma^{-1}$-Lipschitz continuous, where
$\gamma^{-1}=\alpha^{-1}+\beta^{-1}$, and thus 
$h\in\mathscr{C}_{1/\gamma}^{1,1}(\HH)$ and it is convex. Hence, 
since $\gamma^{-1}=\alpha^{-1}+\beta^{-1}> \rho+\beta^{-1}\ge \rho$, 
it 
follows from 
Proposition~\ref{p:BH} that, for every $x$ and $y$ in $\HH$,
\begin{align}
\scal{x\!-\!y\!}{\!\nabla\phi(x)\!-\!\nabla\phi(y)}\!
&=\!\scal{x\!-\!y\!}{\!\nabla h(x)\!-\!\nabla 
h(y)}\!-\!\rho\|x-y\|^2\le (\gamma^{-1}-\rho)\|x-y\|^2,\nonumber
\end{align} 
which yields $\phi\in\mathscr{C}_{\gamma^{-1}-\rho}^{1,1}(\HH)$
in view of Proposition~\ref{p:BH}. In addition, we have
\begin{equation}
\label{e:gradphi}
G_{\tau\nabla h}
=\Id-\tau(\nabla\phi+\rho\Id)
=(1-\tau\rho)\Id-\tau\nabla\phi.
\end{equation}
Now let 
 $\tau\in\left]0,2\beta\alpha/(\beta+\alpha)\right[=\left]0,2\gamma\right[$
  and denote 
$p=G_{\tau\nabla h}\,x$ and $q=G_{\tau \nabla h}\,y$. Since
$\phi\in\mathscr{C}_{\gamma^{-1}-\rho}^{1,1}(\HH)$ and it is convex, it 
follows from 
\eqref{e:gradphi}, Proposition~\ref{p:BH}, and
$\nabla\phi\in\mathcal{C}_{(\gamma^{-1}-\rho)^{-1}}$ that
\begin{align}
\|p-q\|^2&=(1-\tau\rho)^2\|x-y\|^2 
+\tau^2\|\nabla\phi(x)-\nabla\phi(y)\|^2
-2\tau(1-\tau\rho)\scal{x-y}{\nabla\phi(x)-\nabla\phi(y)}\nonumber\\
&\le(1-\tau\rho)^2\|x-y\|^2
+\tau\big(\tau(\gamma^{-1}+\rho)
-2\big)\scal{x-y}{\nabla\phi(x)-\nabla\phi(y)}\nonumber\\
&\le(1-\tau\rho)^2\|x-y\|^2
+\tau\max\{0,\tau(\gamma^{-1}+\rho)
-2\}(\gamma^{-1}-\rho)\|x-y\|^2\nonumber\\
&=
\|x-y\|^2\max\{(1-\tau\rho)^2,(1-\tau\gamma^{-1})^2\}\nonumber
\end{align}
and we obtain \eqref{e:constantG2strongopt}. 

\ref{p:stronglysplit1optii}: Let $\tau\in\left]0,2\alpha\right[$. It follows 
from \ref{p:stronglysplit1opti} that,
in the case when $g=0$ ($\beta^{-1}=0$), $G_{\tau\nabla f}$ is 
$r_{T_1}(\tau)$-Lipschitz continuous, where $r_{T_1}(\tau)$ is defined 
in 
\eqref{e:constgradreduced}. The result follows from
$T_{\tau\nabla g,\tau\nabla f}
=\prox_{\tau g}\circ G_{\tau\nabla f}$
and the nonexpansivity of $\prox_{\tau g}$.

\ref{p:stronglysplit1optiiv}: 
Let $\tau\in\left]0,2\beta\right]$. 
 We deduce from 
Proposition~\ref{p:3}\eqref{p:3i} that $J_{\tau\nabla f}=\prox_{\tau f}$
is $r_{T_2}(\tau)$-Lipschitz continuous, where 
$r_{T_2}(\tau)$ is defined in \eqref{e:consFB2}.
The result follows from $T_{\tau\nabla f,\tau\nabla g}
=\prox_{\tau f}\circ G_{\tau\nabla g}$ and the nonexpansivity 
of $G_{\tau\nabla g}$. 

\ref{p:stronglysplit1optiii}: See \cite[Theorem~2]{BoydGiss}.
\ref{p:stronglysplit1optiv}: It is a consequence of
\cite[Theorem~2]{BoydGiss} and \cite[Theorem~5.6]{Giselsson2017}
in the particular case when $\mathcal{A}=\nabla f$ and 
$\mathcal{B}=\nabla g$.

In all the cases, the minimum is obtained via simple 
computations.
\end{document}